\documentclass[12pt]{amsart}
\usepackage{amsmath}
\usepackage{amscd}
\usepackage{amsthm}
\usepackage{amssymb} \usepackage{latexsym}
\usepackage{eufrak}
\usepackage{euscript}
\usepackage{epsfig}
\usepackage{graphics}
\usepackage{array}
\usepackage{enumerate}
\usepackage{dsfont}
\usepackage{color}

\newcommand{\bel}[1]{\begin{equation}\label{#1}}

\newcommand{\be}{\begin{equation}}

\newcommand{\ba}{\begin{eqnarray}}
\newcommand{\ea}{\end{eqnarray}}
\newcommand{\rf}[1]{(\ref{#1})}

\newcommand{\qe}{\end{equation}}
\newcommand{\R}{{\mathbb R}}
\newcommand{\N}{{\mathbb N}}

\newtheorem{thesis}{Thesis}
\newcommand{\btl}[1]{\begin{thesis}\label{#1}}
\newcommand{\et}{\end{thesis}}

\theoremstyle{theorem}
\newtheorem{theo}{Theorem}[section]
\newtheorem{satz}{Proposition}[section]
\theoremstyle{corollary}
\newtheorem{coro}{Corollary}[section]
\theoremstyle{lemma}
\newtheorem{lemma}{Lemma}[section]
\theoremstyle{definition}
\newtheorem{defi}{Definition}[section]
\theoremstyle{proof}

\theoremstyle{remark}
\newtheorem*{rem}{Remark}
\theoremstyle{definition}
\newtheorem{example}{Example}
\newcommand{\g}{\mathbb{G}}
\newcommand{\gu}{\mathbb{G}^{\mathrm{u}}}
\newcommand{\gp}{\mathbb{G}^+}
\newcommand{\gpu}{\mathbb{G}^{\mathrm{u}+}}

\newcommand{\gcf}{\mathbb{G}^{\mathrm{ac}}}
\newcommand{\D}{\Delta}

\begin{document}
\title{Normalized graph Laplacians for directed graphs}
\author{Frank Bauer}
\address{Max Planck Institute for Mathematics in the Sciences\\
Leipzig 04103, Germany.} \email{Frank.Bauer@mis.mpg.de}

\begin{abstract}
We consider the normalized Laplace operator for directed graphs
with positive and negative edge weights. This generalization of
the normalized Laplace operator for undirected graphs is used to
characterize directed acyclic graphs. Moreover, we identify
certain structural properties of the underlying graph with
extremal eigenvalues of the normalized Laplace operator. We prove
comparison theorems that establish a relationship between the
eigenvalues of directed graphs and certain undirected graphs. This
relationship is used to derive eigenvalue estimates for directed
graphs. Finally we introduce the concept of neighborhood graphs
for directed graphs and use it to obtain further eigenvalue
estimates.
\end{abstract}
\maketitle \textbf{To appear in: Linear Algebra and its
Applications}
 \tableofcontents
\section{Introduction}
For undirected graphs with nonnegative weights, the normalized
graph Laplace operator $\D$ is a well studied object, see
e.~g.~the monograph \cite{Chung97}. In addition to its
mathematical importance, the spectrum of the normalized Laplace
operator has various applications in chemistry and physics.
However, it is not always sufficient to study the normalized
Laplace operator for undirected graphs with nonnegative weights.
In many biological applications, one naturally has to consider
directed graphs with positive and negative weights \cite{Bauer10}.
For instance, in a neuronal network only the presynaptic neuron
influences the postsynaptic one, but not vice versa. Furthermore,
the synapses can be of inhibitory or excitatory type. Inhibitory
and excitatory synapses enhance or suppress, respectively, the
activity of the postsynaptic neuron and thus the directionality of
the synapses and the existence of excitatory and inhibitory
synapses crucially influence the dynamics in neuronal networks
\cite{Bauer10}. Hence, a realistic model of a neuronal network has
to be a directed graph with positive and negative weights in which
the neurons correspond to the vertices and the excitatory and
inhibitory synaptic connections are modelled by directed edges
with positive and negative weights, respectively.

In contrast to undirected graphs not much is known about
normalized Laplace operators for directed graphs. In
\cite{Chungdirected} Chung studied a normalized Laplace operator
for strongly connected directed graphs with nonnegative weights.
This Laplace operator is defined as a self-adjoint operator using
the transition probability operator and the Perron
vector\footnote{A similar construction is used in \cite{Wu05} to
study the algebraic connectivity of the Laplace operator $L=D-W$
defined on directed graphs.}. For our purposes, however, this
definition of the normalized Laplace operator is not suitable
since by the above considerations we are particularly interested
in graphs that are neither strongly connected nor have nonnegative
weights.  In this article, we define a novel normalized Laplace
operator that can in particular be defined for directed graphs
that are neither strongly connected nor have nonnegative weights.
In contrast to Chung's normalized Laplace operator our normalized
Laplace operator is in general neither self-adjoint nor
nonnegative. Moreover, our definition of the normalized Laplace
operator is motivated by the observation that it has already found
applications in the field of complex networks, see \cite{BAJ,
Bauer10}.

The paper is organized as follows. In Section $2$ we define the
normalized Laplace operator for directed graphs and in Section $3$
and Section $4$ we derive its basic spectral properties. In
Section $5$ we characterize directed acyclic graphs by means of
their spectrum. Extremal eigenvalues of the Laplace operator are
studied in Section $6$ and Section $7$. In Section $8$ we prove
several eigenvalues estimates for the normalized Laplace operator.
Finally in Section $9$ we introduce the concept of neighborhood
graphs and use it to derive further eigenvalue estimates.
\section{Preliminaries}\label{D71}
Unless stated otherwise, we consider finite simple loopless
graphs. Let $\Gamma=(V,E,w)$ be a weighted directed graph on $n$
vertices where $V$ denotes the vertex set, $E$ denotes the edge
set, and $w:V\times V \rightarrow \R$ is the associated weight
function of the graph. For a directed edge $e=(i,j)\in E$, we say
that there is an edge from $i$ to $j$. The weight of $e=(i,j)$ is
given by $w_{ji}$ \footnote{We use this convention instead of
denoting the weight of the edge $e=(i,j)$ by $w_{ij}$, since it is
more appropriate if one studies dynamical systems defined on
graphs, see for example  \cite{BAJ}.} and we use the convention
that $w_{ji}=0$ if and only if $e=(i,j)\notin E$. The graph
$\Gamma = (V,E,w)$ is an undirected weighted graph if the
associated weight function $w$ is symmetric, i.e. satisfies
$w_{ij} = w_{ji}$ for all $i$ and $j$. Furthermore, $\Gamma$ is a
graph with non-negative weights if the associated weight function
$w$ satisfies $w_{ij}\geq 0$ for all $i$ and $j$. For ease of
notation, let $\mathbb{G}$ denote the class of
 weighted directed  graphs $\Gamma$. Furthermore, let $\gu$,
$\gp$ and $\gpu$ denote  the class of weighted undirected graphs,
the class of weighted directed graphs with non-negative weights
and the class of weighted undirected graphs with non-negative
weights, respectively. The in-degree and the out-degree of vertex
$i$ are given by $d_i^{\mathrm{in}} := \sum_jw_{ij}$ and
$d_i^{\mathrm{out}}:= \sum_jw_{ji}$, respectively. A graph is said
to be \textit{balanced} if $d_i^{\mathrm{in}} =
d_i^{\mathrm{out}}$ for all $i\in V$. Since every undirected graph
is balanced, the two notions coincide for undirected graphs. Thus,
we simply refer to the degree $d_i$ of an undirected graph.  A
graph $\Gamma$ is said to have a \textit{spanning tree} if there
exists a vertex from which all other vertices can be reached
following directed edges. A directed graph $\Gamma$ is
\textit{weakly connected } if replacing all of its directed edges
with undirected edges produces a connected (undirected) graph. A
directed graph $\Gamma$ is \textit{strongly connected} if for any
pair of distinct vertices $i$ and $j$ there exists a path from $i$
to $j$ and a path from $j$ to $i$. An undirected graph is weakly
connected if and only if it is strongly connected. Hence, we do
not distinguish between weakly and strongly connected undirected
graphs. We simply say that the undirected graph is connected if it
is weakly (strongly) connected.

\begin{defi}\label{DLaplace}Let $C(V)$
denote the space of complex valued functions on $V$. The
normalized graph Laplace operator for directed graphs
$\Gamma\in\g$ is defined as
\[\D:C(V)\rightarrow C(V),\]
\begin{eqnarray}
\Delta v(i) = \left\{\begin{array}{c l}v(i) -
\frac{1}{d^{\mathrm{in}}_i}\sum_j w_{ij} v(j) &\mbox{if } d^{\mathrm{in}}_i \neq 0 .\\
0& \mbox{else}.
\end{array} \right.
\end{eqnarray}
\end{defi}
If $d^{\mathrm{in}}_i \neq 0$ for all $i\in V$, then $\D$ is given
by
\[\D = I - D^{-1}W,\] where $D: C(V) \to C(V)$ is the multiplication
operator defined by \be \label{D83} Dv(i) = d_i^\mathrm{in}v(i)\qe
and $W: C(V) \to C(V)$ is the weighted adjacency operator
\[Wv(i) = \sum_{j\in V}w_{ij}v(j).\]
 When restricted to
undirected graphs with nonnegative weights, Definition
\ref{DLaplace} reduces to the well-known definition of the
normalized Laplace operator for undirected graphs with nonnegative
weights, c.f.\cite{Jost01}.

The choice of normalizing by the in-degree is to some extend
arbitrary. One could also consider the operator
\[\overline{\D}:C(V)\rightarrow C(V),\]
\begin{eqnarray}
\overline{\D} v(i) = \left\{\begin{array}{c l}v(i) -
\frac{1}{d^{\mathrm{out}}_i}\sum_j w_{ji} v(j) &\mbox{if } d^{\mathrm{out}}_i \neq 0 .\\
0& \mbox{else}.
\end{array} \right.
\end{eqnarray}
Note however, that both operators $\D$ and $\overline{\D}$ are
equivalent to each other in the sense that $\D(\Gamma)=
\overline{\D}(\overline{\Gamma})$, where $\overline{\Gamma}$ is
the graph that is obtained from $\Gamma$ by reversing all edges.

Since we consider a normalized graph Laplace operator, i.~e.~we
normalize the edge weights w.r.t. the in-degree, vertices with
zero in-degree are of particular interest and need a special
treatment. We define the following:
\begin{defi}\label{D57} We say that
vertex $i$ is in-isolated or simply isolated if $w_{ij} = 0$ for
all $j\in V$. Similarly, vertex $i$ is said to be
in-quasi-isolated or simply quasi-isolated if $d^{\mathrm{in}}_i
=\sum_jw_{ij} = 0$.
\end{defi}Note that every isolated vertex is quasi-isolated but not
vice versa. These definitions can be extended to induced
subgraphs:
\begin{defi}\label{D72} Let $\Gamma=(V,E,w)\in\g$ be a graph and $\Gamma^\prime =
(V^\prime, E^\prime, w^\prime)$ be an induced subgraph of
$\Gamma$, i.e.  $V^\prime \subseteq V$,
$E^\prime=E\cap(V^\prime\times V^\prime) \subseteq E$, and
$w^\prime:V^\prime\times V^\prime\rightarrow \mathds{R}$,
$w^\prime := w|_{E^\prime}$. We say that $\Gamma^\prime$ is
isolated if $w_{ij} = 0$ for all $i\in V^\prime$ and $j \notin
V^\prime$. Similarly, $\Gamma^\prime$ is said to be quasi-isolated
if $\sum_{j\in V\setminus V^\prime} w_{ij} = 0$ for all  $i\in
V^\prime$.
\end{defi}
We do not exclude the case where $V^\prime = V$. Thus, in
particular, every graph $\Gamma$ is isolated.

It is useful to introduce the reduced Laplace operator $\D_R$.
\begin{defi}
Let $V_R\subseteq V$ be the subset of all vertices that are not
quasi-isolated. The reduced Laplace operator $\D_R: C(V_R)\to
C(V_R)$ is defined as \be \label{D67} \D_R v(i) = v(i) -
\frac{1}{d_i^\mathrm{in}} \sum_{j\in V_R} w_{ij} v(j) \quad i\in
V_R,\qe where $d_i^\mathrm{in}$ is the in-degree of vertex $i$ in
$\Gamma$.
\end{defi}
As above $\D_R$ can be written in the form $\D_R = I_R -
D_R^{-1}W_R$ where $I_R$ is the identity operator on $V_R$.

It is easy to see that the spectrum of $\D$ consists of the
eigenvalues of $\D_R$ and $|V\setminus V_R|$ times the eigenvalue
$0$, i.~e.~\be \label{D5} \mathrm{spec}(\D) = (|V\setminus V_R|
\mbox{ times the eigenvalue } 0)\cup \mathrm{spec}(\D_R).\qe

We remark here that $\D_R$ can be considered as a Dirichlet
Laplace operator. The Dirichlet Laplace operator for directed
graphs is defined as in the case of undirected graphs, see
e.~g.~\cite{Grigoryan09}. Let $\Omega\subseteq V$ and denote by
$C(\Omega)$ the space of complex valued functions $v:\Omega
\rightarrow \mathds{C}$. The Dirichlet Laplace operator
$\D_\Omega$ on $C(\Omega)$ is defined as follows: First extend $v$
to the whole of $V$ by setting $v=0$ outside $\Omega$ and then
\[\D_\Omega v = (\D v)|_\Omega,
\] i.~e.~for any $i\in\Omega$ we have \[\D_\Omega v(i) = v(i) -
\frac{1}{d^{\mathrm{in}}_i} \sum_{j\in V} w_{ij}v(j)= v(i) -
\frac{1}{d^{\mathrm{in}}_i} \sum_{j\in\Omega} w_{ij}v(j)\] since
$v(j) = 0$ for all $j\in V\setminus \Omega$. Hence,
$\D_R=\D_\Omega$ if we set $\Omega = V_R$.

As already mentioned in the introduction, we are particularly
interested in graphs that are not strongly connected. However,
every graph that is not strongly connected can uniquely be
decomposed into its strongly connected components
\cite{Brualdi91}. Using this decomposition, the Laplace operator
$\D$ can be represented in the Frobenius normal form
\cite{Brualdi91}, i.~e.~either $\Gamma$ is strongly connected or
there exists an integer $z>1$ s.t.
\begin{equation} \label{D1}\D  = \left(
\begin{array}{cccccc} \D_{1} & \D_{12} &...  &\D_{1z} &
\\0& \D_{2}&...&\D_{2z} \\\vdots & \vdots& \ddots & \vdots \\  0& 0& ...&
\D_{z}
\end{array} \right),
\end{equation}
where $\D_{1},\ldots \D_{z}$ are square matrices corresponding to
the strongly connected components $\Gamma_{1},\ldots,\Gamma_{z}$
of $\Gamma$. In the following, the vertex set of $\Gamma_k$ is
denoted by $V_k$. Then the off-diagonal elements of $\D_k$ are of
the form $\frac{w_{ij}}{d^{\mathrm{in}}_i}$ for all $i,j\in V_k$
if $d^{\mathrm{in}}_i\neq 0$ and zero otherwise and the diagonal
elements are either zero (if the in-degree of the corresponding
vertex is equal to zero) or one (if the in-degree of the
corresponding vertex is nonzero). If $V_k$ does not contain a
quasi-isolated vertex, then $\D_k$ is irreducible. Furthermore,
the submatrices $\D_{kl}$, $1\leq k <l\leq z$ are determined by
the connectivity structure between different strongly connected
components. For example, $\D_{kl}$ contains all elements of the
form $\frac{w_{ij}}{d^{\mathrm{in}}_i}$ for all $i\in V_k$ and all
$j\in V_l$. A simple consequence of \rf{D1} is that \be
\label{D15} \mathrm{spec}(\D) =
\bigcup_{i=1}^z\mathrm{spec}(\D_{i}).\qe Note that $\D_i$,
$i=1,\ldots,z$, is a matrix representation of the Dirichlet
Laplace operator of the strongly connected component $\Gamma_i$ ,
i.e. $\D_i = \D_\Omega$ for $\Omega = V_i$. To sum up our
discussion, the spectrum of the Laplace operator of a directed
graph is the union of the spectra of the Dirichlet Laplace
operators of its strongly connected components $\Gamma_i$.

We conclude this section by  introducing the operator $P := I
-\D$. We have \[P:C(V)\rightarrow C(V),\]
\begin{eqnarray}
P v(i) = \left\{\begin{array}{c l}
\frac{1}{d^{\mathrm{in}}_i}\sum_j w_{ij} v(j) &\mbox{if } d^{\mathrm{in}}_i \neq 0 .\\
v(i) & \mbox{else}.
\end{array} \right.
\end{eqnarray} For technical reasons, it is sometimes convenient to study
$P$ instead of $\D$. Clearly, the eigenvalues of $\D$ and $P$ are
related to each other by \be \label{D14} \lambda(\D) = 1 -
\lambda(P), \qe i.~e.~if $\lambda$ is an eigenvalue of $P$ then
$1-\lambda$ is an eigenvalue of $\D$. When restricted to graphs
$\Gamma\in\gp$, $P(\Gamma)$ is equal to the transition probability
operator of the reversal graph $\overline{\Gamma}$. Furthermore,
we define the reduced operator $P_R = I_R - \D_R = D^{-1}_RW_R$.

\section{Basic properties of the spectrum}In this section, we
collect basic spectral properties of the Laplace operator $\D$.
\begin{satz}\label{DBasic}Let $\Gamma\in\g$ then following
assertions hold:
\begin{enumerate}
\item[(i)] The Laplace operator $\D$ has always an eigenvalue
$\lambda_0 = 0$ and the corresponding eigenfunction is given by
the constant function.

\item[(ii)] The eigenvalues of $\D$ appear in complex conjugate
pairs. \item[(iii)]The eigenvalues of $\D$ satisfy
\[\sum_{i=0}^{n-1}\lambda_i = \sum_{i=0}^{n-1}\Re(\lambda_i) =
|V_R|.\] \item[(iv)] The spectrum of $\D$ is invariant under
multiplying all weights of the form $w_{ij}$ for some fixed $i$
and $j = 1,..,n$ by a non-zero constant $c$. \item [(v)] The
spectrum of $\D$ is invariant under multiplying all weights by a
non-zero constant $c$. \item [(vi)] The Laplace operator spectrum
of a graph is the union of the Laplace operator spectra of its
weakly connected components.
\end{enumerate}
\end{satz}
\begin{proof}
\begin{itemize}
\item[$(i)$] This follows immediately from the definition of $\D$
since \[ \Delta v(i) = \left\{\begin{array}{c l}
\frac{1}{d^{\mathrm{in}}_i}\sum_j w_{ij} (v(i)-v(j)) &\mbox{if } d^{\mathrm{in}}_i \neq 0 .\\
0& \mbox{else}.
\end{array} \right.\]
\item[$(ii)$] Since $\D$ can be represented as a real matrix, the
characteristic polynomial is given by
\[\det(\D-\lambda I) = a_0 + a_1\lambda + ... +
a_{n-1}\lambda^{n-1},\] with $a_i\in\mathds{R}$ for all $i =
0,1,\ldots, n-1$.  Consequently, $\det(\D-\lambda I) = 0$ if and
only if $\det(\D-\overline{\lambda}I) = 0$. \item[$(iii)$] The
equality $\sum_{i=0}^{n-1} \lambda_i=
\sum_{i=0}^{n-1}\Re(\lambda_i)$ follows from $(ii)$. By
considering the trace of $\D$, one obtains
$\sum_{i=0}^{n-1}\lambda_i = |V_R|$. \item[$(iv)$,]$(v)$ and
$(vi)$ follow directly from the definition of $\D$.
\end{itemize}
\end{proof}From Proposition \ref{DBasic}$(v)$ it follows that it is equivalent
to study the spectrum of graphs with nonnegative or nonpositive
weights. Moreover, because of Proposition \ref{DBasic}$(vi)$, we
will restrict ourselves to weakly connected graphs in the
following.
\begin{satz}\label{D16}
The spectrum of $\D$ satisfies
\[\mathrm{spec}(\D) \subseteq \mathcal{D}(1,r_1)\cup \{0\} \subseteq \mathcal{D}(1,r_2)\cup \{0\}
\subseteq \mathcal{D}(1,r)\cup \{0\},\]where $\mathcal{D}(c,r)$
denotes the disk in the complex plane centered at $c$ with radius
$r$ and
 \[r_1 := \max_{p=1,\ldots, z}\max_{i\in V_{R,p}}\frac{\sum_{j\in V_{R,p}}|w_{ij}|}{|d^{\mathrm{in}}_i|}, \]
\[r_2 := \max_{i\in V_R}\frac{\sum_{j\in V_R}|w_{ij}|}{|d^{\mathrm{in}}_i|},\]and
\be \label{D66} r :=\max_{i\in V}r(i),\qe where $r(i)
=\frac{\sum_{j\in V}|w_{ij}|}{|d^{\mathrm{in}}_i|}$. Here,
$V_{R,1},\ldots, V_{R,z}$ are the strongly connected components of
the induced subgraph $\Gamma_R$ whose vertex set is given by
$V_R$. We use the convention that $r_1, r_2$ and $r$ are equal to
zero if $d^{\mathrm{in}}_i=0$.
\end{satz}
\begin{proof}Clearly, $r_1\leq r_2\leq r$ and the proof follows from Gersgorin's circle
theorem (see e.~g.~\cite{Horn90}) and \rf{D5}-\rf{D15}.
\end{proof}
For undirected graphs with nonnegative weights Proposition
\ref{D16} reduces to the well-known result \cite{Chung97}, that
all eigenvalues of $\D$ are contained in the interval $[0,2]$.

The radius $r$ in Proposition \ref{D16} has the following
properties: $r\geq 1$ if and only if $V_R\neq \emptyset$ and $r=0$
if and only if $V_R = \emptyset$.

\begin{lemma}
Let $\Gamma$ be a graph without quasi-isolated vertices and let
$r(i) = r=1$ for all $i\in V$. Then there exists a graph
$\Gamma^+\in \gp$ that is isospectral to $\Gamma$.
\end{lemma}
\begin{proof}
Since $r=1$ it follows from the definition of $r$ that for every
vertex $i\in V$ the sign $\mathrm{sgn}(w_{ij})$ is the same for
all $j\in V$. By Proposition \ref{DBasic} $(iv)$ the graph
$\Gamma^+\in \gp$ that is obtained from $\Gamma$ by replacing the
associated weight function $w$ by its absolute value $|w|$ is
isospectral to $\Gamma$.
\end{proof} In the following, $\Gamma^+$ is called the associated
positive graph of $\Gamma$.
\begin{coro} \label{D35} For graphs $\Gamma\in\g$ the
nonzero eigenvalues satisfy \be \label{D32}
 1-r\leq \min_{i:\lambda_i\neq 0}\Re(\lambda_i)\leq
\frac{|V_R|}{n-m_0}\leq \max_{i:\lambda_i\neq 0}
\Re(\lambda_i)\leq 1+r,\qe where $m_0$ denotes the multiplicity of
the eigenvalue zero. In particular, we have \[1\leq
\max_{i:\lambda_i\neq 0}\Re(\lambda_i).\]
\end{coro}
\begin{proof}
This estimate follows from Proposition \ref{DBasic} $(iii)$ and
Proposition \ref{D16}. The last statement follows from the
observation that $n-m_0\leq |V_R|$.
\end{proof}
Later, in Corollary \ref{D23}, we characterize all graphs for
which $\max_{i:\lambda_i\neq 0} \Re(\lambda_i)= 1+r$. Similarly,
in Corollary \ref{D17}, we characterize all graphs for which
$\min_{i:\lambda_i\neq 0}\Re(\lambda_i)= 1-r$, provided that
$r>1$.

For graphs with nonnegative weights, Proposition \ref{D16} can be
further improved.
\begin{satz}\label{D84} Let $\Gamma\in\gp$, then all eigenvalues of the
Laplace operator $\D$ are contained in the shaded region in Figure
\ref{DFig.1}.
\end{satz}
\begin{figure}\begin{center}
\includegraphics[width =
5.3cm]{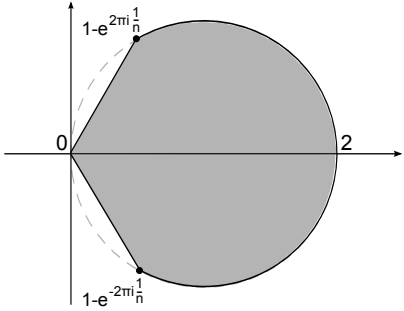}\caption{\label{DFig.1} For a graph
$\Gamma\in\gp$ with $n$ vertices, all eigenvalues of $\D$ are
contained in the shaded region. }
\end{center}
\end{figure}
\begin{proof}This follows from the results
in \cite{Dmitriev45}, see \cite{Minc} for further discussion.
\end{proof}We close this section by considering the following example.
\begin{example}\label{D82}In \cite{Chung97} it is shown
that the smallest non-trivial eigenvalue $\lambda_1$ of
non-complete undirected graphs $\Gamma\in\gpu$ with nonnegative
weights satisfies $\lambda_1\leq 1$. It is tempting to conjecture
that $\min_{i\neq 0} \Re(\lambda_i)\leq 1$ for all non-complete
undirected graphs with positive and negative weights and for all
non-complete directed graphs with nonnegative weights. However,
the two examples in Figure \ref{DFig.3} show that this is, in
general, not true. For both, the non-complete graph $\Gamma_1\in
\gu$ in Figure \ref{DFig.3} (a) and the non-complete graph
$\Gamma_2\in\gp$ in Figure \ref{DFig.3} (b) we have $\min_{i\neq
0} \Re(\lambda_i)> 1$. Thus, there exist non-complete graphs
$\Gamma_1\in\gu$ and $\Gamma_2\in\gp$ for which the smallest
non-zero real part of the eigenvalues is larger than the smallest
non-zero eigenvalue of all non-complete graphs $\Gamma\in\gpu$.
This observation has interesting consequences for the
synchronization of coupled oscillators, see \cite{BauerAtay}.
\end{example}
\begin{figure}\begin{center}
\includegraphics[width =
6.3cm]{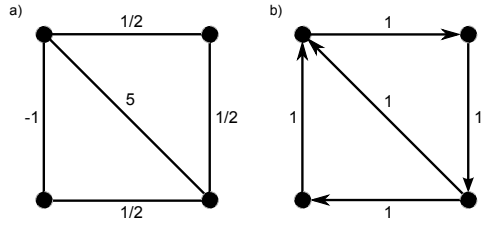}\caption{\label{DFig.3} a) The eigenvalues of
$\D$ are $1.45 \pm 0.46 i, 1.10, 0$. b) The eigenvalues of $\D$
are $1.65, 1.18\pm 0.86 i, 0$.}
\end{center}
\end{figure}
\section{Spectrum of $\D$ and isolated components of $\Gamma$}
We have the following simple observation:
\begin{lemma}\label{D8} Consider a graph $\Gamma\in\g$ and
let $\Gamma_i, 1\leq i\leq r$ be its strongly connected
components. Furthermore, let the Laplace operator $\D$ be
represented in Frobenius normal form \rf{D1}. Then,
\begin{itemize}
\item[$(i)$]If $\Gamma_i$ is isolated then $\D_{ij}=0$ for all
$j>i$. \item[$(ii)$]If $\Gamma_i$ is quasi-isolated then the row
sums of $\D_{i,(i+1)}\ldots\D_{ir}$ add up to zero.
\end{itemize}
Moreover, if $\Gamma\in \gp$ then
\begin{itemize}
\item[$(iii)$]$\Gamma_i$ is isolated if and only if $\D_{ij}=0$
for all $j>i$. \item[$(iv)$]$\Gamma_i$ is quasi-isolated if and
only if the row sums of $\D_{i,(i+1)}\ldots\D_{ir}$ add up to
zero.
\end{itemize}
\end{lemma}
\begin{lemma}\label{D73}
Every graph $\Gamma\in\g$ contains at least one isolated strongly
connected component. Furthermore, $\Gamma\in\g$ contains exactly
one isolated strongly connected component if and only if $\Gamma$
contains a spanning tree.
\end{lemma}
\begin{proof} This follows immediately from the Frobenius
 normal form of $\D$.
\end{proof}

In particular, every undirected graph $\Gamma\in\gu$ is strongly
connected and isolated.

In general, it is not true that the spectrum of an induced
subgraph $\Gamma^\prime$ of $\Gamma$ is contained in the spectrum
of the whole $\Gamma$, i.~e.~$\mathrm{spec}(\D(\Gamma'))
\nsubseteq \mathrm{spec}(\D(\Gamma))$. However, we have the
following result:
\begin{satz}\label{D9} Let $\Gamma \in\g$ and $\Gamma^\prime$
be an induced subgraph of $\Gamma$. If one of the following
conditions is satisfied
\begin{itemize}\item[(i)] $\Gamma^\prime$ consists of $1\leq p\leq r$ strongly
connected components of $\Gamma$ and is quasi-isolated,
\item[(ii)] $\Gamma^\prime$ is isolated,
\end{itemize}
then
\[\mathrm{spec}(\D(\Gamma')) \subseteq
\mathrm{spec}(\D(\Gamma)).\]
\end{satz}
\begin{proof}$(i)$
First, assume that $\Gamma^\prime$ is quasi-isolated and consists
of $p$ strongly connected components of $\Gamma$. Without loss of
generality  we assume that $\Gamma^\prime = \cup_{i=1}^p\Gamma_i$.
Since $\Gamma^\prime$ is quasi-isolated we have for all vertices
$i\in V^\prime$:
\[d^{\mathrm{in}}_i = \sum_{j\in V}w_{ij} = \sum_{j\in V^\prime}w_{ij}+\sum_{j\in V\setminus V^\prime}w_{ij}
=\sum_{j\in V^\prime}w_{ij}\] Thus, the in-degree of each vertex
$i\in V^\prime$ is not affected by the vertices in $V\setminus
V^\prime$. Using \rf{D1} and  \rf{D15} we obtain
\[ \mathrm{spec}(\D(\Gamma^\prime)) =
\bigcup_{i=1}^p\mathrm{spec}(\D_i)\subseteq
\bigcup_{i=1}^r\mathrm{spec}(\D_i) =\mathrm{spec}(\D(\Gamma)).\]

$(ii)$ Now assume that $\Gamma^\prime$ is isolated. Observe that
each isolated induced subgraph $\Gamma^\prime$ of $\Gamma$ has to
consist of $p$, $1\leq p\leq r$ strongly connected components of
$\Gamma$. Thus, the second assertion follows from the first one.
\end{proof}
We will make use of the following theorem by Taussky
\cite{Taussky49}.
\begin{theo}[\cite{Taussky49}]\label{D12}
A complex $n\times n$ matrix $A$  is non-singular if $A$ is
irreducible and $|A_{ii}|\geq \sum_{j\neq i}|A_{ij}|$ with
equality in at most $n-1$ cases.
\end{theo}
\begin{lemma}\label{D10}
Let $\Gamma \in\gp$ be a graph with nonnegative weights and let
$\Gamma_i$, $1\leq i\leq r$ be its strongly connected components.
Furthermore, let $\D$ be represented in Frobenius normal form.
Then, zero is an eigenvalue (in fact a simple eigenvalue) of
$\D_i$ if and only if $\Gamma_i$ is isolated.
\end{lemma}
\begin{proof}We observe that since $\Gamma\in \gp$, it follows that $d_j^\mathrm{in}\neq
0$ for all $j\in \Gamma_i$ and hence $\D_i$ is irreducible. First
assume that $\Gamma_i$ is not isolated. Assume further that
$\Gamma_i$ consists of more than one vertex. Then there exists a
vertex $k\in V_i$ s.t. $w_{kl}\neq 0$ for some $l\notin V_i$. For
vertex $k$ we have
\[ |(\Delta_{i})_{kk}|=1>\frac{\sum_{j\in V_i}|w_{kj}|}{\sum_{j\in
V}|w_{kj}|} = \sum_{j\in V_i}\frac{|w_{kj}|}{|d^{\mathrm{in}}_k|}
= \sum_{j\in V_i}|(\D_i)_{kj}|.\] For all other $j\in V_i$ we have
\[|(\D_i)_{jj}|=1 \geq \sum_{l\in
V_i}\frac{|w_{jl}|}{|d^{\mathrm{in}}_j|}= \sum_{l\in
V_i}|(\D_i)_{jl}|\] and hence by Theorem \ref{D12}, $0$ is not an
eigenvalue of $\D_i$. If $\Gamma_i$ consists of one vertex, then
$1$ is the only eigenvalue of $\D_i$ and hence $0$ is not an
eigenvalue of $\D_i$.

Now we assume that $\Gamma_i$ is isolated and consists of more
than one vertex. We consider the operator $P_i:= I_i - \D_i$,
where $I_i$ is the identity operator on $\Gamma_i$. Since all row
sums of $P_i$ are equal to one, it follows that the spectral
radius $\rho$ of $P_i$ is equal to one. Moreover, since $\Gamma\in
\gp$, it follows that $P_i$ is non-negative and irreducible. The
Perron-Frobenius theorem implies that $\rho=1$ is a simple
eigenvalue of $P_i$ and hence, by \rf{D14}, $0$ is a simple
eigenvalue of $\D_i$. If $\Gamma_i$ is an isolated vertex, then
clearly $0$ is a simple eigenvalue of $\D_i$.
\end{proof}

\begin{theo}\label{D40}For a graph $\Gamma \in\gp$
the following four statements are equivalent:
\begin{enumerate}\item[$(i)$] The multiplicity $m_1(P)$ of the eigenvalue one of $P$ is equal to $k$.
\item[$(ii)$] The multiplicity $m_0(\D)$ of the eigenvalue zero of
the Laplace operator  $\D$ is equal to $k$. \item[$(iii)$] There
exist $k$ isolated strongly connected components in $\Gamma$.
\item[$(iv)$] The minimum number of directed trees needed to span
the whole graph is equal to $k$.
\end{enumerate}
\end{theo}
\begin{proof}
$(i) \Leftrightarrow (ii)$ follows from (\ref{D14}). $(ii)
\Leftrightarrow (iii)$ follows from Lemma \ref{D10} and \rf{D15}.
$(iii) \Leftrightarrow (iv)$  follows from the Frobenius normal
form and Lemma \ref{D8} ($iii$).
\end{proof}
A similar result was obtained for the algebraic graph Laplace
operator $L=D-W$ in \cite{Wu}. In the presence of negative
weights, Theorem \ref{D40} is not true anymore. However, for
general graphs $\Gamma\in\g$ we have the following:
\begin{coro}\label{D70}
For a graph $\Gamma\in\g$ we have: \begin{itemize}
\item[$(i)$]$m_1(P)=m_0(\D)$. \item[$(ii)$] The number of isolated
strongly connected components in $\Gamma$ is equal to the minimum
number of directed trees needed to span $\Gamma$. \item[$(iii)$]
The number of isolated strongly connected components in $\Gamma$
is less or equal to the multiplicity of the eigenvalue zero of
$\D$.
\end{itemize}
\end{coro}
\begin{proof} The first two statements follows exactly in the same way as in
Theorem \ref{D40}, since the proof is not affected by the presence
of negative weights. The third assertion follows from the
observation that for every isolated strongly connected component
$\Gamma_i$ the Laplace operator $\D_i$ has at least one eigenvalue
equal to zero. This observation follows immediately from
Proposition \ref{D9} and Proposition \ref{DBasic} $(i)$.
\end{proof}

\section{Directed acyclic graphs}
\begin{defi}\label{D74}
A \textit{directed cycle} is a cycle with all edges being oriented
in the same direction. A vertex is a \textit{cyclic vertex} if it
is contained in at least one directed cycle. A graph is an
\textit{directed acyclic graph} if none of its vertices are
cyclic. The class of all directed acyclic graphs is denoted by
$\gcf$.
\end{defi}Note that a directed acyclic graph is not necessarily a
directed tree, because we do not exclude the existence of
topological cycles in the graph. If $\D$ is represented in the
Frobenius normal form, then we immediately obtain the following:
\begin{lemma}\label{Duppertriangular}The following three statements
are equivalent: \begin{itemize} \item[(i)] $\Gamma\in\gcf$ is a
directed acyclic graph.\item[(ii)] Every strongly connected
component of $\Gamma$ consists of exactly one vertex. \item[(iii)]
$\Delta$ represented in Frobenius normal form is upper triangular.
\end{itemize}
\end{lemma}
\begin{theo}\label{D33} \begin{itemize}\item[]\item[(i)]
If $\Gamma \in \gcf$ is a directed acyclic graph, then
$\mathrm{spec}(\D) \subseteq\{0,1\}$. Furthermore, $m_0(\D) =
|V\setminus V_R|$ and $m_1(\D) = |V_R|$. \item[(ii)] $\Gamma \in
\gp$ and $\mathrm{spec}(\D) \subseteq\{0,1\}$ if and only if
$\Gamma \in \g^{ac,+}$.
\end{itemize}
\end{theo}
\begin{proof}
The first part follows immediately from Lemma
\ref{Duppertriangular}, the definition of $\D$, and \rf{D15}.
Thus, we only have to prove that if $\Gamma \in \gp$ and
$\mathrm{spec}(\D) \subseteq\{0,1\}$ then $\Gamma \in \g^{ac,+}$.
Assume the converse, i.e. assume that $\Gamma \in \gp$ and
$\mathrm{spec}(\D) \subseteq\{0,1\}$ but $\Gamma \notin
\g^{ac,+}$. Then, by Lemma \ref{Duppertriangular} there exists a
strongly connected component $\Gamma_i$ in $\Gamma$ consisting of
at least two vertices.  First, assume that $\Gamma_i$ is isolated.
Then, by Lemma \ref{D10} exactly one eigenvalue of $\D_i$ is equal
to zero. Using Proposition \ref{D9} and Corollary \ref{D35} we
conclude that there exists an eigenvalue
$\lambda\in\mathrm{spec}(\D)$ s.t. $\Re(\lambda) \geq
\frac{n_i}{n_i-1}>1$ where $n_i=|V_i|>1$. This is the desired
contradiction.  Now assume that $\Gamma_i$ is not isolated. By
Lemma \ref{D10}, all eigenvalues of $\D_i$ are non-zero.  Since
$\Gamma\in\gp$, $P_i$ is non-negative and irreducible. The
Perron-Frobenius theorem implies that the spectral radius $\rho$
of $P_i$ is positive and is an eigenvalue of $P_i$. By \rf{D14},
$1-\rho$ is an eigenvalue of $\D_i$ that satisfies $1>1-\rho>0$.
Hence, we have a contradiction to the assumption that
$\mathrm{spec}(\D) \subseteq\{0,1\}$.
\end{proof}
\begin{coro}\label{D75}
If $k$ eigenvalues of $\D$ are not equal to $0$ or $1$, then there
exists at least $k$ cyclic vertices in the graph.
\end{coro}

\section{Extremal eigenvalues} In this section, we study eigenvalues $\lambda$ of
$\D$ that satisfy $|1-\lambda|=r$, i.~e.~eigenvalues that are
boundary points of the disc $\mathcal{D}(1,r)$ in Proposition
\ref{D16}.
\begin{defi}
Let $\Gamma\in\g$ and $\Gamma^\prime$ be an induced subgraph of
$\Gamma$. The induced subgraph $\Gamma^\prime$ is said to be
maximal if all vertices $i\in V^\prime$ satisfy\[r(i) = \max_l
r(l)=r,\] where as before
\[r(i) := \frac{\sum_{j\in V}|w_{ij}|}{|d^{\mathrm{in}}_i|}.\]\end{defi}
Note that, if we exclude isolated vertices, then  every graph with
nonnegative weights $\Gamma\in\gp$ is maximal. Thus, in
particular, every connected graph $\Gamma\in\gpu$ is maximal.
\begin{satz}\label{D18}
Let $\lambda\neq 0$ be an eigenvalue of $\D$ that satisfies
$|1-\lambda|=r$. Then $\Gamma$ possesses a maximal, isolated,
strongly connected component that consists of at least two
vertices.
\end{satz}
Before we prove Proposition \ref{D18}, we consider the following
lemma.
\begin{lemma}\label{D58}
Let $\lambda\neq 1$, be an eigenvalue of $P$ that satisfies
$|\lambda|=r$. Then $\lambda$ is an eigenvalue of the Dirichlet
operator $P_k$ that corresponds to the strongly connected
component $\Gamma_k$ for some $k$. Furthermore, $\Gamma_k$
consists of at least two vertices and the corresponding
eigenfunction $u$ for $\lambda$ satisfies $|u(i)|= \mathrm{const}$
for all $i\in V_k$.
\end{lemma}
\begin{proof}From \rf{D1} and  \rf{D15} it follow that $\lambda$ is an eigenvalue
of $P_k$ for some $1\leq k\leq z$. Since we assume that
$\lambda\neq 1$ it follows that $V_R\neq \emptyset$ and hence
$r\geq 1$. This in turn implies that $\Gamma_k$ consists of at
least two vertices because otherwise by Theorem \ref{D33} and
\rf{D14}, $P_k$ has only one eigenvalue which is either equal to
zero or one. So we only have to prove that $|u(i)|=
\mathrm{const}$  for all $i\in V_k$.

Assume that $|u|$ is not constant on $V_k$. Since $\Gamma_k$ is
strongly connected, there exists two vertices $i,j$ in $V_k$ that
satisfy $w_{ij}\neq 0$ and $|u(j)|<|u(i)|= \max_{l\in V_k}|u(l)|$.
Again, since $\lambda\neq 1$ it follows that $i\in V_R$ and hence
we have

\begin{eqnarray*}
\left|P_ku(i)\right| &=&
\left|\frac{1}{d^{\mathrm{in}}_i}\sum_{l\in V_k} w_{il}u(l)\right|
\leq \frac{1}{|d^{\mathrm{in}}_i|}\sum_{l\in V_k} |w_{il}||u(l)|
\\&<&r(i)\max_{l\in
V_k}|u(l)| \leq r \max_{l\in V_k}|u(l)|.
\end{eqnarray*}On the other hand we have \be
\left|P_ku(i)\right| = |\lambda||u(i)| = r \max_{l\in
V_k}|u(l)|.\qe This is a contradiction to the last equation.
\end{proof}
Now we prove Proposition \ref{D18}.
\begin{proof}
For simplicity, we consider $P$ instead of $\D$. Formulated in
terms of $P$ we have to show the following: Let $\lambda\neq 1$ be
an eigenvalue of $P$ that satisfies $|\lambda|=r$ then $\Gamma$
possesses an isolated, maximal, strongly connected component
consisting of at least two vertices. As in the proof of Lemma
\ref{D58} one can show that $\lambda$ is an eigenvalue of the
operator $P_k$ that corresponds to a strongly connected component
$\Gamma_k$ consisting of at least two vertices.

First we show that all vertices in $\Gamma_k$ are not
quasi-isolated. Assume that at least one vertex, say vertex $l$,
in $\Gamma_k$ is quasi-isolated. Then \[P_ku(l)= u(l) = \lambda
u(l).\] Since $\lambda\neq 1$ it follows that $u(l)=0$. Thus, we
have $|u(l)|<\max_{j\in V_k}|u(j)|$ which is a contradiction to
Lemma \ref{D58}.

Now we prove that $\Gamma_k$ is isolated. Assume that $\Gamma_k$
is not isolated, then there exists a vertex $i\in V_k$ and a
neighbor $j\notin V_k$ of $i$. Thus, we have for the vertex $i$
that
\begin{eqnarray*}
\left|P_ku(i)\right| &=&
\left|\frac{1}{d^{\mathrm{in}}_i}\sum_{l\in V_k}
w_{il}u(l)\right|\leq \frac{1}{|d^{\mathrm{in}}_i|}\sum_{l\in V_k}
|w_{il}||u(l)|\\&<& \frac{1}{|d^{\mathrm{in}}_i|}\sum_{l\in V}
|w_{il}|\max_{l \in V_k} |u(l)| = r(i)\max_{l\in V_k}|u(l)|
\\&\leq& r \max_{l\in
V_k}|u(l)|.
\end{eqnarray*}
On the other hand, we have \be \label{D45} \left|P_ku(i)\right| =
|\lambda||u(i)|= r|u(i)|.\qe Comparing these two equations yields
\[|u(i)| < \max_{l\in V_k}|u(l)|.\]Again, this is a
contradiction to Lemma \ref{D58}.

Finally, we have to prove that the strongly connected component
$\Gamma_k$ is maximal. Assume that $\Gamma_k$ not maximal. Then
there exists a vertex, say $i\in V_k$, such that $r(i)<r$. We
conclude that
\begin{eqnarray*}|P_ku(i)| &=& \left|\frac{1}{d^{\mathrm{in}}_i}\sum_{l\in V_k}
w_{il}u(l)\right| \leq \frac{1}{|d^{\mathrm{in}}_i|}\sum_{l\in
V_k}|w_{il}||u(l)|\\&\leq& \frac{1}{|d^{\mathrm{in}}_i|}\sum_{l\in
V}|w_{il}|\max_{l\in V_k}|u(l)|= r(i)\max_{l\in V_k}|u(l)|
\\&<& r\max_{l\in V_k}|u(l)|.
\end{eqnarray*} Together with \rf{D45} this implies that $|u(i)|
<\max_{l\in V_k}|u(l)|$. Again, this is a contradiction to Lemma
\ref{D58}.
\end{proof}
In Proposition \ref{D18} we have to exclude the eigenvalue
$\lambda=0$. However, if we assume that all vertices are not
quasi-isolated, Proposition \ref{D18} also holds for $\lambda=0$.
\begin{satz}\label{D22}
Let $\Gamma\in\g$ and assume that all vertices are not
quasi-isolated. If $\lambda =0$ is an eigenvalue of $\D$ that
satisfies $|1-\lambda|=r$, then there exists a maximal, isolated,
strongly connected component consisting of at least two vertices
in $\Gamma$.
\end{satz}
\begin{proof}Since $V=V_R$ we have for all $i\in V$ that $r(i)\geq 1$.
By assumption, we have $1=r$ and hence $r(i) =1 $ for all $i\in
V$. This implies that every strongly connected component in
$\Gamma$ is maximal. By Lemma \ref{D73} every graph contains an
isolated strongly connected component. Since $\lambda=0$ and we
exclude quasi-isolated vertices it follows that there exists an
isolated maximal strongly connected component in $\Gamma$ that
consists of at least two vertices.
\end{proof}
\section{$k$-partite graphs and anti-$k$-partite graphs}
\subsection{$k$-partite graphs}
\begin{defi}
$\Gamma \in \g$ is $k$-partite, $k\geq 2$, if $d^{\mathrm{in}}_i
\neq 0$ for all $i\in V$ and the vertex set $V$ consists of $k$
nonempty subsets $V_1,\ldots, V_k$ such that the following holds:
There are only edges from vertices  $j\in V_{q-1}$ to vertices
$i\in V_{q}$, $q=1,\ldots,k$,  if
$\frac{w_{ij}}{d^{\mathrm{in}}_i}>0$ and if $k$ is even from
vertices $j\in V_{q+l}$ to vertices $i\in V_q$, $q=1,\ldots,k$, if
$\frac{w_{ij}}{d^{\mathrm{in}}_i}<0$ where $l=\frac{k}{2}-1\in \N$
and we identify $V_{k+1}$ with $V_1$.
\end{defi}
The condition $l=\frac{k}{2}-1\in \N$ implies that, in a
$k$-partite graph, there can only exists weights satisfying
$\frac{w_{ij}}{d^{\mathrm{in}}_i}<0$ if $k$ is even. The special
choice of $l$ ensures that the distance between different
neighbors of one particular vertex, say vertex $i$, is a multiple
of $\frac{k}{2}$. If the distance of two neighbors $s,t$ of $i$ is
an odd multiple of $\frac{k}{2}$, then $s,t$ belong to different
subsets and $\frac{w_{is}}{w_{it}}<0$. If the distance between
$s,t$ is an even multiple of $\frac{k}{2}$, then $s,t$ belong to
the same subset and $\frac{w_{is}}{w_{it}}>0$.

\begin{theo}\label{D28}
$\Gamma\in\g$ contains a $k$-partite isolated maximal strongly
connected component if and only if $1-re^{\pm 2\pi i \frac{1}{k}}$
are eigenvalues of $\D$.
\end{theo}
\begin{proof} Again, for technical reasons, we consider $P$ instead of $\D$.
Since the eigenvalues appear in complex conjugate pairs
(Proposition \ref{DBasic} (ii)), it is sufficient to show that
$re^{2 \pi i \frac{1}{k}}$ is an eigenvalue of $P$. Assume that
$\Gamma$ contains a $k$-partite isolated maximal strongly
connected component $\Gamma_p$. We claim that the function
\[u^1(j) = \left\{\begin{array}{c}e^{2\pi i \frac{k}{k}}\\e^{2\pi i
\frac{k-1}{k}}\\\vdots\\e^{2\pi i \frac{1}{k}}
\end{array}\right.\begin{array}{c}\mbox{ if }  j \in
V_{p,1}\\ \mbox{ if }  j \in V_{p,2}\\\vdots\\
\mbox{ if }  j \in V_{p,k},
\end{array}\]where $V_{p,1},\ldots,V_{p,k}$
is a $k$-partite decomposition of $V_p$, is an eigenfunction for
the eigenvalue $re^{2\pi i \frac{1}{k}}$ of $P_p$. For any $j\in
V_{p,q}$, $1\leq q\leq k$, we have
\begin{eqnarray*}
P_pu^1(j)&=& \frac{1}{d_j^\mathrm{in}} \sum_{t\in V_p}w_{jt}u^1(t)
\\&=& \frac{1}{d_j^\mathrm{in}}\left( \sum_{t\in
V_{p,q-1}}w_{jt}u^1(t) + \sum_{t\in
V_{p,q+l}}w_{jt}u^1(t)\right)\\&=&
\frac{1}{d_j^\mathrm{in}}\left(\sum_{t\in V_{p,q-1}}w_{jt}e^{2\pi
i \frac{1}{k}} u^1(j) + \sum_{t\in V_{p,q+l}}w_{jt}e^{-\pi
i}e^{2\pi i \frac{1}{k}} u^1(j)\right)\\&=&
\frac{1}{|d^{\mathrm{in}}_j|}\sum_{t\in V_{p,q-1}}|w_{jt}|e^{2\pi
i \frac{1}{k}}u^1(j) -\frac{1}{|d^{\mathrm{in}}_j|}\sum_{t\in
V_{p,q+l}}|w_{jt}|e^{-\pi i}e^{2\pi i \frac{1}{k}}u^1(j) \\&=&
\frac{1}{|d^{\mathrm{in}}_j|}\sum_{t\in V_p}|w_{jt}|e^{2\pi i
\frac{1}{k}}u^1(j)
=\frac{1}{|d^{\mathrm{in}}_j|}\sum_{t\in V}|w_{jt}|e^{2\pi i \frac{1}{k}}u^1(j)\\
&=& r(j) e^{2\pi i \frac{1}{k}}u^1(j) = r e^{2\pi i
\frac{1}{k}}u^1(j),
\end{eqnarray*}
where we used that the $k$-partite component $\Gamma_p$ is
isolated and maximal. We conclude that $r e^{2\pi i \frac{1}{k}}$
is an eigenvalue of $P_p$ and, by Proposition \ref{D9}, $r e^{2\pi
i \frac{1}{k}}$ is an eigenvalue of $P$.

Now assume that $re^{2\pi i \frac{1}{k}}$ is an eigenvalue of $P$.
Since $|re^{2\pi i \frac{1}{k}}|=r$ and $re^{2\pi i
\frac{1}{k}}\neq 1$, Proposition \ref{D18} implies that $\Gamma$
contains an isolated maximal strongly connected component
$\Gamma_p$ and $re^{2\pi i \frac{1}{k}}$ is an eigenvalue of the
corresponding Dirichlet operator $P_p$. We only have to prove that
$\Gamma_p$ is $k$-partite.

Let $u\in C(V_p)$ be an eigenfunction for the eigenvalue $r
e^{2\pi i \frac{1}{k}}$. On the one hand, since $\Gamma_p$ is
maximal and isolated, all $j \in V_p$ satisfy
\begin{eqnarray} P_pu(j) &=& re^{2\pi i \frac{1}{k}}u(j) =
\frac{1}{|d^{\mathrm{in}}_j|}\sum_{t\in V}|w_{jt}|e^{2\pi i
\frac{1}{k}}u(j)  \\&=& \frac{1}{|d^{\mathrm{in}}_j|}\sum_{t\in
V_p}|w_{jt}|e^{2\pi i \frac{1}{k}}u(j).
\end{eqnarray}
On the other hand \begin{equation} P_pu(j) =
\frac{1}{d^{\mathrm{in}}_j}\sum_{t\in
V_p}w_{jt}u(t).\end{equation} Comparing these two equations yields
\begin{equation}\label{D24} \sum_{t\in V_p}\frac{|w_{jt}|}{|d^{\mathrm{in}}_j|}
=\sum_{t\in
V_p}\frac{w_{jt}}{d^{\mathrm{in}}_j}\frac{u(t)}{u(j)}e^{-2\pi i
\frac{1}{k}}.
\end{equation} Lemma \ref{D58} implies that the
eigenfunction $u$ satisfies $|u(t)| = |u(j)|$ for all $j,t\in
V_p$. Thus, $\frac{u(t)}{u(j)}e^{-2\pi i \frac{1}{k}}$ is a
complex number whose absolute value is equal to one. Since we
consider only real weights, we have equality in \rf{D24} if \be
\label{D26} u(j) = e^{-2\pi i \frac{1}{k}}u(t),\qe whenever
$\frac{w_{jt}}{d^{\mathrm{in}}_j}>0$ and\[u(j) =-e^{-2\pi i
\frac{1}{k}}u(t)= e^{-\pi i}e^{-2\pi i \frac{1}{k}}u(t),\]
whenever $\frac{w_{jt}}{d^{\mathrm{in}}_j}<0$.

First, assume that $\frac{w_{ij}}{d^{\mathrm{in}}_i}>0$ for all
edges in $\Gamma_p$. If $t$ is a neighbor of $j$ then the
eigenfunction has to satisfy equation \rf{D26}. Since $\Gamma_p$
is strongly connected we can uniquely assign to each vertex $i$ a
value $u(i)$ such that every $k$-th vertex in a directed path has
the same value since $\left(e^{-2\pi i \frac{1}{k}}\right)^k = 1$.
Now decompose the vertex set into $k$ non-empty subsets s.t. all
vertices with the same $u$-value belong to the same subset of
$V_p$. This yields a $k$-partite decomposition of $\Gamma_p$.

If there also exist edges s.t.
$\frac{w_{ij}}{d^{\mathrm{in}}_i}<0$ is satisfied, then the
crucial observation is that if
$\frac{w_{jt}}{d^{\mathrm{in}}_j}<0$ for some $j$ and $t$ then
there has to exist another neighbor $s$ of $j$ s.t.
$\frac{w_{js}}{d^{\mathrm{in}}_j}>0$. We conclude that every
vertex $j$ has at least one neighbor $s$ such that
$\frac{w_{js}}{d_j^\mathrm{in}}>0$. Thus, there exist $k$
different $u$-values and we can find a $k$-partite decomposition
of $V_p$ similarly as in the case studied before.
\end{proof}
Even if we do not require that the $k$-partite component is
maximal we have:
\begin{coro}
Let $\Gamma\in\g$ contain a $k$-partite isolated strongly
connected component $\Gamma_p$, and let $r(j)=c$ for all $j\in
V_p$ and some constant $c$. Then $1-ce^{\pm 2\pi i \frac{1}{k}}$
are eigenvalues of $\D$.
\end{coro}

Theorem \ref{D28} can be used to characterize the graph $\Gamma\in
\gp$ whose spectrum contains the distinguished eigenvalues
$1-e^{\pm2\pi i \frac{1}{n}}$ in Figure \ref{DFig.1}. As a special
case of Theorem \ref{D28} we obtain:
\begin{coro}Let $\Gamma\in\gp$ be a graph with $n$ vertices. Then, $1-e^{\pm2\pi i
\frac{1}{n}}$ is an eigenvalue of $\D(\Gamma)$ iff $\Gamma$ is a
directed cycle.
\end{coro}

\begin{defi}
The associated positive graph $\Gamma^+\in \gp$ of a graph
$\Gamma\in \g$ is obtained from $\Gamma$ by replacing every weight
$w_{ij}$ by its absolute value $|w_{ij}|$. The eigenvalues of
$\Gamma^+$ are denoted by $\lambda^+_0,\ldots \lambda_{n-1}^+$and
the Laplace operator defined on the graph $\Gamma^+$ is denoted by
$\D^+$.
\end{defi}
Clearly, a graph $\Gamma \in \gp$ with nonnegative weights
coincides with its associated positive graph, i.~e.~$\Gamma =
\Gamma^+$.
\begin{rem}
It is also possible to define the associated negative graph
$\Gamma^-$ of a graph $\Gamma$ that is obtained from $\Gamma$ by
replacing every weight $w_{ij}$ by $-|w_{ij}|$. Note however, that
by Proposition \ref{DBasic} $(v)$ the graphs $\Gamma^-$ and
$\Gamma^+$ are isospectral. Thus, we will only consider $\Gamma^+$
in the following.
\end{rem}
\begin{theo}\label{D27}
Let $\Gamma \in\g$ be a $k$-partite graph  and $r(j) = r$ for all
$j \in V$. Then, the spectra of $\D^+$ and $\D$ satisfy the
following relation: $\lambda^+\in \mathrm{spec}(\D^+)$ iff
$1-re^{\pm 2\pi i \frac{1}{k}}(1-\lambda^+)\in \mathrm{spec}(\D)$.
\end{theo}
\begin{proof}Let the function $u$ satisfy $\D^+u = \lambda^+u$. We
define a new function $v$ in the following way:
\begin{equation} v(j) = \left\{
\begin{array}{r c l} e^{2\pi i \frac{1}{k}} u(j) & \quad if \quad & j \in V_1 \\
 e^{2\pi i \frac{2}{k}}u(j) &\quad if \quad & j \in
V_2\\ \vdots &&\\ e^{2\pi i \frac{k}{k}} u(j) &\quad if \quad & j
\in V_k,
\end{array} \right.
\end{equation} where $V_1,\ldots, V_k$ is a $k$-partite decomposition
of $V$.  We show that $v$ is an eigenfunction for $\D$ and the
corresponding eigenvalue is given by $(1 - r e^{-2\pi i
\frac{1}{k}}(1-\lambda^+))$. For any $j \in V_q$ and $1\leq q\leq
k$, we have

\begin{eqnarray*}\D v(j) &=& v(j) - \frac{1}{d_j^\mathrm{in}}\sum_{t\in
V_{q-1}}w_{jt}v(t) - \frac{1}{d_j^\mathrm{in}}\sum_{t\in
V_{q+l}}w_{jt}v(t)
\\ &=&e^{2\pi i \frac{q}{k}}u(j) - \frac{1}{|d_j^\mathrm{in}|} \sum_{t
\in V_{q-1}} |w_{jt}| e^{2\pi i \frac{q-1}{k}} u(t) +
\frac{1}{|d_j^\mathrm{in}|} \sum_{t \in V_{q+l}} |w_{jt}|
e^{\pi i}e^{2\pi i \frac{q-1}{k}}u(t)\\
&=& e^{2\pi i \frac{q}{k}}u(j) -
\frac{1}{|d^{\mathrm{in}}_j|}\sum_{t\in V} |w_{jt}|
e^{2\pi i \frac{q-1}{k}}u(t)  \\
&=& e^{2\pi i \frac{q}{k}}u(j) - re^{2\pi i \frac{q-1}{k}}u(j) +
re^{2\pi i \frac{q-1}{k}}\underbrace{\left( u(j) -
\frac{1}{\sum_{t\in V} |w_{jt}|}\sum_{t\in
V}|w_{jt}|u(t)\right)}_{= \D^+u(j) = \lambda^+u(j)}
\\ &=& (1 - r e^{-2\pi i \frac{1}{k}}(1-\lambda^+))v(j).
\end{eqnarray*}
Since the edge weights are real, $\D$ can be represented as a real
matrix and hence $\overline{v}$ is an eigenfunction for the
eigenvalue $1 - r e^{2\pi i \frac{1}{k}}(1-\lambda^+))$.

The other direction follows in a similar way. To be more precise,
for an eigenfunction $v$ of $\D$  we define the function $u$ by
\begin{equation} u(j) = \left\{
\begin{array}{r c l} e^{2\pi i \frac{1}{k}} v(j) & \quad if \quad & j \in V_1 \\
 e^{2\pi i \frac{2}{k}}v(j) &\quad if \quad & j \in
V_2\\ \vdots &&\\ e^{2\pi i \frac{k}{k}} v(j) &\quad if \quad &
j\in V_k.
\end{array} \right.
\end{equation}As above, one can show that $u$ is an
eigenfunction for $\D^+$ and corresponding eigenvalue
$1-\frac{1}{r}e^{-2\pi i \frac{1}{k}}(1-\lambda)$.
\end{proof}
Note that in Theorem \ref{D27} we do not assume that $\Gamma$ is
strongly connected. However, if we assume in addition that
$\Gamma$ is strongly connected, then we have the following result:
\begin{coro}
Let  $\Gamma \in\g$ be a strongly connected graph, and $r(j) = r$
for all $j\in V$. Then, $\Gamma$ is $k$-partite if and only if the
spectra of $\D^+$ and $\D$ satisfy the following: $\lambda^+$ is
an eigenvalue of $\D^+$ iff  $1-re^{\pm 2\pi i
\frac{1}{k}}(1-\lambda^+)$ is an eigenvalue of $\D$.
\end{coro}
\begin{proof}
One direction follows from Theorem \ref{D27}. The other direction
follows from the observation that zero is an eigenvalue of $\D^+$
and thus $1-re^{\pm 2\pi i \frac{1}{k}}$ is an eigenvalue of $\D$.
Since $\Gamma$ is strongly connected, it follows from Theorem
\ref{D28} that the whole graph is $k$-partite.
\end{proof}
Moreover, a $k$-partite graph has the following eigenvalues:
\begin{satz}\label{D31}
Let $\Gamma\in \g$ be a $k$-partite graph and $r(l) = r$ for all
$l \in V$. Then, $1-re^{2 \pi i \frac{m}{k}} \in
\mathrm{spec}(\D)$ for $1 \leq m \leq k-1$ and $m$ odd. If, in
addition, $\frac{w_{jt}}{d^{\mathrm{in}}_j}>0$ for all $j,t\in V$,
then $1-e^{2 \pi i \frac{m}{k}} \in \mathrm{spec}(\D)$ for all
$0\leq m\leq k-1$.
\end{satz}
\begin{proof}
In order to prove that $1-re^{2 \pi i \frac{m}{k}}$ is an
eigenvalue of $\D$, it is sufficient to show that $re^{2 \pi i
\frac{m}{k}}$ is an eigenvalue of $P$. Consider the functions
\begin{equation} u^m(j) = \left\{
\begin{array}{r c l} e^{2\pi i \frac{mk}{k}}  & \quad if \quad & j \in V_1 \\
 e^{2\pi i \frac{m(k-1)}{k}} &\quad if \quad & j \in
V_2\\ \vdots &&\\ e^{2\pi i \frac{m1}{k}} &\quad if \quad & j \in
V_k,
\end{array} \right.
\end{equation}for  $m = 0,1,\ldots k-1$.
One easily checks that these functions are linearly independent if
$k>2$.

For all $j\in V_q$, $q=1,\ldots,k$,  and $0\leq m\leq k-1$ we have
\begin{eqnarray}
Pu^m(j)&=& \nonumber \frac{1}{d_j^\mathrm{in}} \sum_{t \in
V_{q-1}}w_{jt}u^m(t) + \frac{1}{d_j^\mathrm{in}} \sum_{t \in
V_{q+l}}w_{jt}u^m(t)\\&=& \nonumber \frac{1}{|d^{\mathrm{in}}_j|}
\sum_{t \in V_{q-1}}|w_{jt}|e^{2\pi i \frac{m}{k}}u^m(j) -
\frac{1}{|d^{\mathrm{in}}_j|}
\sum_{t \in V_{q+l}}|w_{jt}|e^{-2\pi i \frac{ml}{k}}u^m(j)\\
&=&\label{D25}
 \frac{1}{|d^{\mathrm{in}}_j|} \sum_{t \in V_{q-1}}|w_{jt}|e^{2\pi i
\frac{m}{k}}u^m(j) - \frac{1}{|d^{\mathrm{in}}_j|} \sum_{t \in
V_{q+l}}|w_{jt}|e^{2\pi i \frac{m}{k}}e^{-\pi i m}u^m(j).
\end{eqnarray} If $m$ is odd, then $e^{-\pi i m}=-1$ and thus
\begin{eqnarray*}Pu^m(j)&=& \frac{1}{|d^{\mathrm{in}}_j|}\sum_{t\in V_{q-1}\cup
V_{q+l}}|w_{jt}|e^{2\pi i
\frac{m}{k}}u^m(j)\\&=&\frac{1}{|d^{\mathrm{in}}_j|}\sum_{t\in
V}|w_{jt}|e^{2\pi i \frac{m}{k}}u^m(j) \\&=& r e^{2\pi i
\frac{m}{k}}u^m(j).\end{eqnarray*} Hence, $1-re^{2\pi i
\frac{m}{k}}$ for $1\leq m\leq k-1$ and $m$ odd is an eigenvalue
of $\D$.

If in addition $\frac{w_{jt}}{d^{\mathrm{in}}_j}>0$ for all $j$
and $t$ in $V$, then $r =1$ and there are only edges from vertices
in $V_{q-1}$ to vertices in $V_q$. Thus, the second term on the
r.h.s. of \rf{D25} vanishes and we can conclude that
\begin{eqnarray*}Pu^m(j)&=& \frac{1}{|d^{\mathrm{in}}_j|}\sum_{t\in
V_{q-1}}|w_{jt}|e^{2\pi i
\frac{m}{k}}u^m(j)=\frac{1}{|d^{\mathrm{in}}_j|}\sum_{t\in
V}|w_{jt}|e^{2\pi i \frac{m}{k}}u^m(j) \\&=& r e^{2\pi i
\frac{m}{k}}u^m(j)= e^{2\pi i \frac{m}{k}}u^m(j).\end{eqnarray*}
This shows that $1-e^{2\pi i \frac{m}{k}}$ for all $m = 0, \ldots,
k-1$ is an eigenvalue of $\D$.
\end{proof}
\subsection{Anti-$k$-partite graphs}
In this section, we study graphs that are closely related to
$k$-partite graphs. We call those graphs anti-$k$-partite graphs
since they have the same topological structure as $k$-partite
graphs but compared to $k$-partite graphs, the normalized weights
$\frac{w_{ij}}{d^{\mathrm{in}}_i}$ in anti-$k$-partite graphs have
always the opposite sign.

\begin{defi}\label{D81}
$\Gamma \in \g$ is anti-$k$-partite, for $k\geq 2$ and $k$ even,
if $d^{\mathrm{in}}_i \neq 0$ for all $i\in V$ and the vertex set
$V$ consists of $k$ nonempty subsets $V_1,\ldots, V_k$ such that
the following holds: There are only edges from vertices  $j\in
V_{q-1}$ to vertices $i\in V_{q}$ if
$\frac{w_{ij}}{d^{\mathrm{in}}_i}<0$ or from vertices $j\in
V_{q+l}$ to vertices $i\in V_q$ if
$\frac{w_{ij}}{d^{\mathrm{in}}_i}>0$ where $l=\frac{k}{2}-1\in \N$
and we identify $V_{k+1}$ with $V_1$.
\end{defi}
In contrast to $k$-partite graphs, anti-$k$-partite graphs can
only be defined if $k$ is even. This follows from the observation
that every vertex $i$ has at least one neighbor $j$ such that
$\frac{w_{ij}}{d^{\mathrm{in}}_i}>0$. Hence, every vertex $i\in
V_q$ has at least one neighbor in $V_{q+l}$ for $q =1,\ldots,k$.
Since we require that $l=\frac{k}{2}-1\in \N$, it follows that $k$
has to be even.

We mention the following simple observation without proof:
\begin{satz}\label{D47}
Let $\Gamma\in\gp$ be an anti-$k$-partite graph and $k=2+4^m$,
where $m=0,1,\ldots$. Then, $\Gamma$ is disconnected and if $m\geq
1$ $\Gamma$ consists of two $\frac{k}{2}$-partite connected
components.
\end{satz}
\begin{theo}\label{D29} Let
$\Gamma\in\g$ contain an anti-$k$-partite maximal, isolated,
strongly connected component, then  $1+re^{\pm2\pi i
\frac{1}{k}}\in\mathrm{spec}(\D)$. Furthermore, if $1+re^{\pm2\pi
i \frac{1}{k}}\in\mathrm{spec}(\D)$ and one of the following two
conditions is satisfied \begin{itemize}\item[(i)] $k=4^m$ for
$m=1,2,\ldots$ \item[(ii)] $k=2+4^m$ for $m=0,1,\ldots$ and $r>1$,
\end{itemize} then $\Gamma$ contains an anti-$k$-partite isolated
maximal strongly connected component.
\end{theo}
\begin{proof}
Assume that $\Gamma$ contains an anti-$k$-partite maximal,
isolated, strongly connected component. In exactly the same way as
in Theorem \ref{D28} one can show that $1+re^{\pm2\pi i
\frac{1}{k}}$ is an eigenvalue of $\D$. We will omit the details
here. Now let $ 1+re^{\pm2\pi i \frac{1}{k}}$ be an eigenvalue of
$\D$. Note that $1+re^{\pm2\pi i \frac{1}{k}}\neq 0$  and thus, by
Proposition \ref{D18}, $\Gamma$ contains an maximal, isolated,
strongly connected component $\Gamma_p$. Furthermore, we have that
$1+re^{\pm2\pi i \frac{1}{k}}$ is an eigenvalue of $\D_p$. By a
reasoning similar to the one in the proof of Theorem \ref{D28} it
follows that the corresponding eigenfunction for $\D_p$ satisfies
\be \label{D46} u(j) = - e^{2\pi i \frac{1}{k}}u(t)\qe whenever
$\frac{w_{jt}}{d^{\mathrm{in}}_j}>0$ and \be \label{D59} u(j) =
e^{2\pi i \frac{1}{k}}u(t)\qe whenever
$\frac{w_{jt}}{d^{\mathrm{in}}_j}<0$. Now assume that $k=4^m$ and
$\frac{w_{tj}}{d^{\mathrm{in}}_t}>0$ for all $t,j\in V_p$.

Since $\Gamma_p$ is strongly connected, $k=4^m$, and neighbors
have to satisfy equation \rf{D46}, we can uniquely assign to every
vertex an $u$-value such that every $k$-th vertex in a directed
path has the same $u$-value.  Now decompose the vertex set into
$k$ non-empty subsets such that all vertices with the same
$u$-value belong to the same subset. This yields an
anti-$k$-partite decomposition of $\Gamma_p$.

If there also exists edges s.t.
$\frac{w_{jt}}{d^{\mathrm{in}}_j}<0$ is satisfied then, again, the
crucial observation is that if
$\frac{w_{jt}}{d^{\mathrm{in}}_j}<0$ for some $j$ and $t$ then
there also has to exist another neighbor $s$ of $j$ s.t.
$\frac{w_{js}}{d^{\mathrm{in}}_j}>0$. Thus, there exist $k$
different $u$-values. Similar to above, we can find an
anti-$k$-partite decomposition of $V_p$.

If $k=2+4^m$, $m=0,1,\ldots,$ the situation is different. If
$\frac{w_{jt}}{d^{\mathrm{in}}_j}>0$ for all $j$ and $t$, then
$r=1$. In this case, we cannot conclude that there exists an
anti-$k$-partite component since already every $\frac{k}{2}$-th
vertex in a directed path has the same $u$-value, i.e.
$(-1)^\frac{k}{2}(e^{2\pi i \frac{1}{k}})^\frac{k}{2} =1$ for
$k=2+4^m$, $m=0,1,\ldots$. Thus, we crucially need that $r>1$. In
this case, every vertex $i$ has at least one neighbor $j$ such
that $\frac{w_{ij}}{d^{\mathrm{in}}_i}<0$. By \rf{D59} it follows
that there has to exist $k$ different $u$-values. Thus, we can
obtain an anti-$k$-partite component of $\Gamma_p$ in the same way
as before.
\end{proof}

A simple example that shows that the assumption $r>1$ is necessary
if $k = 2+ 4^m$  in the last theorem. If $1-r=0$ is an eigenvalue
of $\Gamma$, then this does not imply that there exists a
$2$-partite isolated maximal strongly connected component in
$\Gamma$.

The next theorem shows that there also exists a relationship
between the spectrum of an anti-$k$-partite graph and its
associated positive graph.
\begin{theo}
Let $\Gamma \in\g$ be an  anti-$k$-partite graph and $r(l) = r$
for all $l \in V$. Then, $\lambda^+\in \mathrm{spec}(\D^+)$ iff
$1+re^{\pm 2\pi i \frac{1}{k}}(1-\lambda^+)\in \mathrm{spec}(\D)$.
\end{theo}We omit the proof of this theorem because it is the same
as the proof of Theorem \ref{D27}.

The next proposition is the corresponding result to Proposition
\ref{D31} in the case of anti-$k$-partite graphs.
\begin{satz}\label{D44}
Let $\Gamma\in \g$ be an anti-$k$-partite graph, s.t. $r(l) = r$
for all $l \in V$, then $1+re^{2 \pi i \frac{m}{k}} \in
\mathrm{spec}(\D)$ for $0 \leq m \leq k-1$, if  $m$ is odd. If in
addition $\frac{w_{ij}}{{d_i^\mathrm{in}}}>0$ for all $i,j \in V$
then $1-e^{2\pi i \frac{m}{k}}\in\mathrm{spec}(\D)$ for $0\leq
m\leq k-1$, $m$ even and $1+e^{2\pi i
\frac{m}{k}}\in\mathrm{spec}(\D)$ for $0\leq m\leq k-1$, $m$ odd.
\end{satz}
\begin{satz}Let $\Gamma\in\g$ be a strongly connected graph and
$r(l)=r$ for all $l\in V$. Assume that $k=4^m$, $m=1,2,\ldots$.
Then, $\Gamma$ is $k$-partite iff $\Gamma$ is anti-$k$-partite.
\end{satz}
\begin{proof} Assume that $\Gamma$ is $k$-partite.
By Proposition \ref{D31}, $1-re^{2\pi i \frac{l}{k}}\in
\mathrm{spec}(\D)$ for $0\leq l\leq k-1$ and $l$ odd. Since $k$ is
of the form $k=4^m$, $\frac{k}{2}+1$ is odd, and so we have
$1-re^{2\pi i \frac{\frac{k}{2}+1}{k}}= 1+re^{2\pi i
\frac{1}{k}}\in \mathrm{spec}(\D)$. From Theorem \ref{D29}, it
follows that $\Gamma$ is anti-$k$-partite. The other direction
follows in the same way by using Proposition \ref{D44} and Theorem
\ref{D28}.
\end{proof}
This proposition shows that if $k=4^m$, $m=1,2,\ldots$ then a
$k$-partite decomposition can be obtained from an anti-$k$-partite
one, and vice versa, by relabelling the vertex sets $V_k$.
\subsection{Special cases: Bipartite and anti-bipartite graphs}
\subsubsection{Bipartite graphs}
As a special case of $k$-partite graphs we obtain:
\begin{defi}\label{D34}
A graph $\Gamma \in\g$ is bipartite (or $2$-partite), if
$d^{\mathrm{in}}_i\neq 0$ for all $i\in V$ and the vertex set $V$
can be decomposed into two nonempty subsets $V_1,V_2$ such that
for neighbors $i$ and $j$ $\frac{w_{ij}}{d^{\mathrm{in}}_i}> 0$ if
$i$ and $j$ belong to different subsets and
$\frac{w_{ij}}{d^{\mathrm{in}}_i}< 0$ if $i$ and $j$ belong to the
same subset.
\end{defi}In the case of undirected graphs
with nonnegative weights, Definition \ref{D34} reduces to the
usual definition of a bipartite graph.
\begin{coro}  \label{D23} A graph $\Gamma\in \g$ contains a
maximal, isolated, bipartite strongly connected component if and
only if $1+r$ is an eigenvalue of $\D$.
\end{coro}Using Corollary \ref{D35} we can reformulate this as follows:
\begin{coro}The spectrum of $\D$ contains the largest possible real eigenvalue
if and only if the graph $\Gamma\in\g$ contains a maximal,
isolated, bipartite strongly connected component.
\end{coro}
For undirected graphs with nonnegative weights, Corollary
\ref{D23} reduces to the well-known result that $\Gamma$ is
bipartite if and only if $2$ is an eigenvalue of $\D$.
\begin{coro}\label{D21} Let $\Gamma \in\g$ be
a bipartite graph and $r(l) = r$ for all $l \in V$. Then,
$\lambda^+\in \mathrm{spec}(\D^+)$ iff $1+r(1-\lambda^+)\in
\mathrm{spec}(\D)$.
\end{coro}
In particular, if $\Gamma\in\gp$ is strongly connected, then
$\Gamma$ is bipartite if and only if with $\lambda$ also
$2-\lambda$ is an eigenvalue of $\D$, i.~e.~the real parts of the
eigenvalues are symmetric about one.
\subsubsection{Anti-bipartite graphs}As a special case of anti-$k$-partite graphs we obtain:
\begin{defi}
A graph $\Gamma \in\g$ is anti-bipartite, if
$d^{\mathrm{in}}_i\neq 0$ for all $i\in V$ and the vertex set $V$
can be decomposed into two nonempty subsets such that for
neighbors $i$ and $j$, $\frac{w_{ij}}{d^{\mathrm{in}}_i}< 0$ if
$i$ and $j$ belong to different subsets and
$\frac{w_{ij}}{d^{\mathrm{in}}_i}> 0$ if $i$ and $j$ belong to the
same subset.
\end{defi}
\begin{lemma}\label{D30} $\Gamma\in\gp$
is anti-bipartite if and only if the graph $\Gamma$ is
disconnected and $d^{\mathrm{in}}_i\neq 0$ for all $i$.
\end{lemma}
\begin{proof}One direction follows from Proposition \ref{D47}.

Now assume that the graph $\Gamma\in\gp$ is disconnected and
$d^{\mathrm{in}}_i\neq 0$ for all $i$. Then there exists at least
two connected components such that
$\frac{w_{ij}}{d^{\mathrm{in}}_i}> 0$ for all neighbors $i$ and
$j$. Distribute the connected components (there exist maybe more
than two) into two nonempty subsets $V_1$ and $V_2$. This is an
anti-bipartite decomposition of the graph.
\end{proof}
\begin{coro}  \label{D17}Let $r>1$, then
$1-r$ is an eigenvalue of $\D$ if and only if the graph contains
an anti-bipartite maximal isolated strongly connected component.
\end{coro}Using Corollary \ref{D35} we can reformulate this as
follows:
\begin{coro}Assume that $r>1$ is satisfied. The spectrum of $\D$ contains the smallest possible real eigenvalue
if and only if the graph $\Gamma\in\g$ contains a maximal,
isolated, anti-bipartite strongly connected component.
\end{coro}
\begin{example}Consider the graph in Figure \ref{DFig.2}. It is easy to calculate
the spectrum of this graph by using the results derived in this
section. First, note that the graph in Figure \ref{DFig.2} is
bipartite and anti-bipartite. Since $r(i) = 3$ for all $i$, we
have $1\pm3\in\mathrm{spec}(\D)$. Zero is always an eigenvalue of
$\D$. The last eigenvalue is equal to $2$ since $\sum_i\lambda_i =
|V_R|=4$. So we have determined all eigenvalues of the graph in
Figure \ref{DFig.2}.
\begin{figure}
\begin{center}\includegraphics[width = 3.0cm]{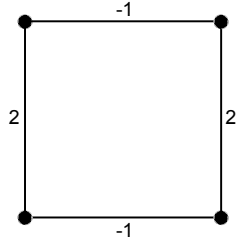}
\caption{\label{DFig.2}  Eigenvalues of $\D$ are $4,2,0,-2$ }
\end{center}
\end{figure}
\end{example}
\section{Bounds for the real and imaginary parts of the
eigenvalues} In this section, we will derive several bounds for
the real and imaginary parts of the eigenvalues of a directed
graph. In the following, we also allow loops in the graph. This
slight generalization is particularly important in the next
section where we introduce the neighborhood graph technique. It is
straightforward to generalize the Laplace operator $\D$ to graphs
with loops. The normalized graph Laplace operator for directed
graphs with loops is defined as:
\[\D:C(V)\rightarrow C(V),\]
\begin{eqnarray}
\Delta v(i) = \left\{\begin{array}{c l}v(i) -
\frac{1}{d^{\mathrm{in}}_i}\sum_j w_{ij} v(j) &\mbox{if } d^{\mathrm{in}}_i \neq 0 \\
0& \mbox{else.}
\end{array} \right.
\end{eqnarray}
The only difference to graphs without loops is that now $w_{ii}$
is not always equal to zero. As for graphs without loops we define
$P= I-\D$. Furthermore, we say that vertex $i$ is in-isolated or
simply isolated if $w_{ij} = 0$ for all $j\in V$. Similarly,
vertex $i$ is said to be in-quasi-isolated or simply
quasi-isolated if $d_i^{\mathrm{in}} = 0$. In particular, an
isolated vertex cannot have a loop. As before, $V_R := \{i\in V :
d_i^\mathrm{in} \neq 0\}$ is the set of all vertices that are not
quasi-isolated.

\subsection{Comparison theorems}\label{D69}

In this section, we show that the real parts of the eigenvalues of
a directed graph can be controlled by the eigenvalues of certain
undirected graphs. Together with well-known estimates for
undirected graphs these comparison results yield estimates the
realparts of the eigenvalues of a directed graph.

We need the following definition:
\begin{defi} \label{D13}
Let $\Gamma\in \g$ be given. The underlying graph
$U(\Gamma)\in\gu$ of $\Gamma$ is obtained from $\Gamma$ by
replacing  each directed edge by an undirected edge of the same
weight. In $U(\Gamma)$ we identify multiple edges between two
vertices with one single edge. The weight of this single edge is
equal to the sum of the weights of the multiple edges.
Furthermore, every loop in $\Gamma$ is replace by a loop of twice
the weight in $U(\Gamma)$.
\end{defi}
Note that the correspondence between directed graphs and their
underlying graphs is not one to one. Indeed, many directed graphs
can have the same underlying graph.

We recall the well-known concept of majorization:
\begin{defi}\label{D11}  Let $a\in\mathbb{R}^{n}$ and $b\in\mathbb{R}^{n}$
be given. If the entries of $a$ and $b$ are arranged in
increasing\footnote{The definition of majorization is not unique
in the literature. Here, we follow the convention in
\cite{Horn90}. In other books, see e.g. \cite{Marshall79},
majorization is defined for vectors arranged in decreasing order.
Reversing the order of the elements has the following consequence:
If $a$ and $b$ are two real vectors whose entries are arranged in
increasing order, and $A$ and $B$ denote the vectors with the same
entries arranged in decreasing order, then $a\prec b$ if and only
if $B\prec A$.} order, then $b$ majorizes $a$, in symbols $a\prec
b$, if
\begin{equation}
\sum_{i=1}^{k}a_{i}\leq\sum_{i=1}^{k}b_{i}\;\;\;\;\;\;\qquad
k=1,...,n-1\label{D122}\end{equation}
 and \begin{equation}
\sum_{i=1}^{n}a_{i}=\sum_{i=1}^{n}b_{i}.\label{D121}\end{equation}
\end{defi}
We will need the following two results:
 \begin{lemma}\label{G7} [R. Rado, see e.g. \cite{Hardy52} p.63 or \cite{Marshall79}]
If $x\prec y$ on $\mathbb{R}^n$ and $a\prec b$ on $\mathbb{R}^m$
then $(x,a)\prec (y,b)$ on $\mathbb{R}^{n+m}$, where $(x,a)$ is
the vector composed of the components of $x$ and $a$ arranged in
increasing order, and similarly for $(y,b)$.
\end{lemma}In particular, Lemma \ref{G7} shows that the
majorization property is preserved if we append the same entries
to both $x$ and $y$ (choose $a=b$ in Lemma \ref{G7}).

In the sequel, let the symmetric part of a matrix $M$ be denoted
by $S(M) := \frac{1}{2}(M + M^\top)$. We make use of a classical
result by Ky Fan \cite{Fan50}:
\begin{lemma}\label{GFan} Let $\lambda(S(M))$ and $\Re[\lambda(M)]$ denote
the column vectors whose components are the eigenvalues of $S(M)$
and the real parts of the eigenvalues of $M$, respectively. If the
components of $\lambda(S(M))$ and $\Re[\lambda(M)]$ are arranged
in increasing order, then for every matrix $M$ we have
\[ \lambda(S(M))\prec\Re[\lambda(M)].\]
\end{lemma}

Using Definition \ref{D13} and Definition \ref{D11} we state the
following comparison result.
\begin{theo}\label{DAB} If $\Gamma\in\g$ is balanced, then
\[\lambda(\D(U(\Gamma))) \prec \Re[\lambda(\D(\Gamma))],\] i.~e.~the
eigenvalues of the underlying graph $U(\Gamma)$ are majorized by
the real parts of the eigenvalues of $\Gamma$.
\end{theo}
\begin{proof}Recall the definition of the reduced Laplace operator
$\D_R = I_R - D^{-1}_RW_R$ in Eq.~\rf{D67}. It is straightforward
to generalize $\D_R$ for graphs with loops. Here however, instead
of $\D_R$ we consider the reduced normalized Laplace operator
$\mathcal{L}_R := I_R-D_R^{-1/2}W_RD_R^{-1/2}$. In the sequel, we
will study matrix representations of $\D_R$ and $\mathcal{L}_R$
that will also be denoted by $\D_R$ and $\mathcal{L}_R$. Since
$D_R^{1/2}$ is nonsingular and
\[\D_R = D_R^{-1/2}\mathcal{L}_RD_R^{1/2},\] it follows that $\mathcal{L}_R$ and $\D_R$ are
similar and hence have the same spectrum. We claim that the
reduced Laplace operator $\mathcal{L}_R$ satisfies
\[S(\mathcal{L}_R(\Gamma)) = \mathcal{L}_R(U(\Gamma)).\]
Since $\Gamma$ is balanced, the degrees of the vertices satisfy
\be\label{D68} 2 d_i^\mathrm{in}(\Gamma) = d_i(U(\Gamma)).\qe
Thus, in particular, the number of quasi-isolated vertices in
$U(\Gamma)$ and $\Gamma$ is the same and so the matrices
$S(\mathcal{L}_R(\Gamma))$ and $ \mathcal{L}_R(U(\Gamma))$ have
the same dimension.

By definition, the diagonal elements satisfy  \begin{eqnarray*}
S(\mathcal{L}_R(\Gamma))_{ii} = 1 -
\frac{w_{ii}}{\sqrt{d_i^\mathrm{in}(\Gamma)d_i^\mathrm{in}(\Gamma)
}}
\end{eqnarray*} and
\[\mathcal{L}_R(U(\Gamma))_{ii}=1 - \frac{2w_{ii}}{\sqrt{d_i(U(\Gamma))d_i(U(\Gamma))}} =
1 -
\frac{2w_{ii}}{\sqrt{2d_i^\mathrm{in}(\Gamma)2d_i^\mathrm{in}(\Gamma)}}
=1-\frac{w_{ii}}{\sqrt{d_i^\mathrm{in}(\Gamma)d_i^\mathrm{in}(\Gamma)
}}\] by \rf{D68}. For the off-diagonal elements, we have
\begin{eqnarray*}
S(\mathcal{L}_R(\Gamma))_{ij} =
-1/2\left(\frac{w_{ij}}{\sqrt{d_i^{\mathrm{in}}(\Gamma)d_j^{\mathrm{in}}(\Gamma)}}
+
\frac{w_{ji}}{\sqrt{d_j^{\mathrm{in}}(\Gamma)d_i^{\mathrm{in}}(\Gamma)}}\right)
\end{eqnarray*}
and
\begin{eqnarray*}\mathcal{L}_R(U(\Gamma))_{ij}&=&-
\frac{w_{ij}+w_{ji}}{\sqrt{ d_i(U(\Gamma))d_j(U(\Gamma))}}
\\&=&
-1/2\frac{w_{ij}+w_{ji}}{\sqrt{d_i^{\mathrm{in}}(\Gamma)
d_j^{\mathrm{in}}(\Gamma)}} ,\end{eqnarray*} where we used
\rf{D68}. This proves our claim. Now it follows that
\[\lambda(\D_R(U(\Gamma)))= \lambda(\mathcal{L}_R(U(\Gamma)))
=\lambda(S(\mathcal{L}_R(\Gamma))) \prec
\Re(\lambda(\mathcal{L}_R(\Gamma)))= \Re(\lambda(\D_R(\Gamma))),\]
where we used Lemma \ref{GFan} and the fact that $\mathcal{L}_R$
and $\D_R$ have the same spectrum. By \rf{D5}, the spectrum of
$\D(\Gamma)$ ($\D(U(\Gamma))$) consists of all eigenvalues of
$\D_R(\Gamma)$ ($\D_R(U(\Gamma))$) and $|V\setminus V_R|$ times
the eigenvalue zero. From \rf{D68} it follows that the number of
quasi-isolated vertices is the same in $U(\Gamma)$ and $\Gamma$.
Hence Lemma \ref{G7} implies $$\lambda(\D(U(\Gamma))) \prec
\Re[\lambda(\D(\Gamma))].$$
\end{proof}

Theorem \ref{DAB} is used in \cite{BauerAtay} to compare the
synchronizability of directed and undirected networks of coupled
phase oscillators.
\\\\
In particular Theorem \ref{DAB} implies:
\begin{coro}\label{D50}For a balanced graph $\Gamma \in \g$ we have
\[\min_{i\neq 0}\lambda_i(\D(U(\Gamma)))\leq \min_{i\neq 0}
\Re(\lambda_i(\D(\Gamma)))\] and \[\max_i
\Re(\lambda_i(\D(\Gamma)))\leq \max_i\lambda_i(\D(U(\Gamma))),\]
where $\lambda_0=0$ is the eigenvalue corresponding to the
constant function.
\end{coro}
Corollary \ref{D50} can now be used to derive explicit bounds for
the real parts of the eigenvalues of a balanced directed graph by
utilizing eigenvalue estimates for undirected graphs. For that
reason, we recall the definition of the Cheeger constant and the
dual Cheeger constant of an undirected graph.

\begin{defi} For an undirected graph the Cheeger constant $h$ is
 defined in the following way \cite{Chung96}:
\bel{ch1}
 h:= \min_{W\subsetneq V}
\frac{|E(W,\overline{W})|}{\min\{\mbox{vol}(W),
\mbox{vol}(\overline{W})\}}, \qe where $W$ and
$\overline{W}=V\setminus W$ yield a partition of the vertex set
$V$ and $W,\overline{W}$ are both nonempty. Here the volume of $W$
is given by $\mbox{vol}(W) := \sum_{i\in W} d_i$. Furthermore,
$E(W,\overline{W})\subseteq E$ is the subset of all edges with one
vertex in $W$ and one vertex in $\overline{W}$, and
$|E(W,\overline{W})|:= \sum_{k\in W, l\in \overline{W}}w_{kl}$ is
the sum of the weights of all edges in $E(W,\overline{W})$.
Similarly, the dual Cheeger constant $\overline{h}$ is defined as
follows \cite{BauerJost}:  For a partition $V_1,V_2,V_3$ of the
vertex set $V$ where $V_1$ and $V_2$ are both nonempty, we define
\be\label{hh}\overline{h}:= \max_{V_1,V_2}
\frac{2|E(V_1,V_2)|}{\mbox{vol}(V_1)+\mbox{vol}(V_2)}.\qe
Although, it seems that $\overline{h}$ does not depend on $V_3$,
$\overline{h}$ is well-defined. In order to see this we note that
for a partition $V_1,V_2$ and $V_3$ of $V$, the volume of $V_i$
can also be written in the form
\begin{equation}
\mbox{vol}(V_i) = \sum_{j=1}^3|E(V_i,V_j)|
\end{equation}
Consequently, $\overline{h}$ is given by
\begin{equation} \overline{h} = \max_{V_1,V_2}
\frac{2|E(V_1,V_2)|}{\sum_{j=1}^3|E(V_1,V_j)| +
\sum_{j=1}^3|E(V_2,V_j)| }\end{equation} and hence depends on
$V_3$.
\end{defi}
It is well known that the Cheeger and the dual Cheeger constant
control the eigenvalues of undirected graphs with nonnegative
weights.
\begin{lemma}\label{Eigenvalues} For an undirected graph  with
nonnegative weights $\Gamma\in \gpu$ we
have:\begin{itemize}\item[(i)]\cite{Chung96} The smallest
nontrivial eigenvalue $\lambda_1$ satisfies
\[1-\sqrt{1-h^2}\leq \lambda_1\leq 2h.\]\item[(ii)]\cite{BauerJost} The largest
eigenvalue $\lambda_{n-1}$ satisfies
\[2\overline{h}\leq \lambda_{n-1}\leq 1+\sqrt{1-(1-\overline{h})^2}.\]
\end{itemize}
\end{lemma}
Combining Lemma \ref{Eigenvalues} with Corollary \ref{D50} we
obtain:

\begin{theo}\label{D19}Let $\Gamma\in\gp$ be a balanced graph, then
\[0\leq 1-\sqrt{1-h^2(U(\Gamma))}\leq \min_{i\neq 0}\Re(\lambda_i(\D(\Gamma)))
\] and \[\max_i\Re(\lambda_i(\D(\Gamma)))\leq
1+\sqrt{1-(1-\overline{h}(U(\Gamma)))^2}\leq 2,\] where
$h(U(\Gamma))$ and $\overline{h}(U(\Gamma))$ are the Cheeger
constant and the dual Cheeger constant of the underlying graph
$U(\Gamma)$.
\end{theo}
\begin{proof}Corollary \ref{D50} implies that
$\min_{i\neq 0}\lambda_i(\D(U(\Gamma)))\leq \min_{i\neq
0}\Re(\lambda_i(\D(\Gamma)))$ and
$\max_i\Re(\lambda_i(\D(\Gamma)))\leq
\max_i\lambda_i(\D(U(\Gamma)))$. Since $U(\Gamma)\in\gpu$, we can
use the estimates in Lemma \ref{Eigenvalues} to control the
eigenvalues of $\D(U(\Gamma))$. This completes the proof.
\end{proof}
Theorem \ref{D19} allows us to interpret the smallest nontrivial
realpart and the largest realpart of the eigenvalues of a balanced
directed graph $\Gamma\in\gp$ in the following way: If the
smallest nontrivial realpart of a balanced directed graph is
small, then it is easy to cut the graph into two large pieces and
if the largest realpart is close to $2$ then the graph is close to
a bipartite one. We illustrate this by considering the following
example.
\begin{example}
We consider the directed cycle $C_n$ of length $n$. Since $C_n$ is
a $n$-partite graph its eigenvalues are given by $1-e^{2\pi i
\frac{k}{n}}$ for $k=0,1,\ldots, n-1$. This implies that
$\min_{i\neq0}\Re(\lambda_i)= 1 - \cos( \frac{2\pi}{n})\rightarrow
0$ as $n \rightarrow \infty$ and $\max_i\Re(\lambda_i) =2$ if $n$
is even and  $\max_i\Re(\lambda_i)
=1-\cos(\frac{n-1}{n}\pi)\rightarrow 2$ if $n$ is odd as
$n\rightarrow \infty$. Since $C_n$ is balanced, Theorem \ref{D19}
implies that it is easy to cut $C_n$ into two large pieces (if $n$
is sufficiently large) and $C_n$ is bipartite if $n$ is even and
close to a bipartite graph if $n$ is sufficiently large and odd.
Indeed, $C_n$ is bipartite if $n$ is even, close to a bipartite
graph if $n$ is odd, and we only have to remove two edges in order
to cut $C_n$ into two large pieces.
\end{example}

Of course, any other eigenvalue estimate than the Cheeger estimate
and the dual Cheeger estimate leads to similar estimates as in
Theorem \ref{D19}. In particular, one can control, $\min_{i\neq
0}\Re(\lambda_i(\D(\Gamma)))$ and
$\max_{i}\Re(\lambda_i(\D(\Gamma)))$ in terms of the diameter
\cite{Chung97, Landau81}, the Olliver-Ricci curvature \cite{BJL}
or arguments involving canonical paths \cite{Diaconis91}.

Now we derive a second comparison theorem that leads to further
eigenvalue estimates. Instead of using the underlying graph
$U(\Gamma)$, we use in the following a different undirected graph
$\widetilde{\Gamma}$ to control the eigenvalues of directed
graphs.

We say that the operator $P=I-\D$ is irreducible if its matrix
representations are irreducible. It is easy to see \cite{Horn90}
that $P$ is irreducible if the graph $\Gamma$ is strongly
connected and $V_R=V$, i.e. $d_i^\mathrm{in} \neq 0$ for all $i$.
If we restrict ourselves to strongly connected graphs with
nonnegative weights, the Perron-Frobenius Theorem \cite{Horn90}
implies that there exists a positive function $\phi$ (i.e.
$\phi(i)>0$ for all $i\in V$) that satisfies \be \label{D80}
\sum_j\frac{w_{ji}}{d_j^\mathrm{in}}\phi(j) = \rho \phi(i) =
\phi(i) \quad \forall i,\qe where $\rho =1$ is the spectral radius
of $P$. The function $\phi$ is sometimes called the Perron vector
of $P$ and is used in the following construction.

\begin{defi}\label{D54} Let $\Gamma=(V,E)\in\gp$ be a strongly connected graph.
The graph $\widetilde{\Gamma}=(V,\widetilde{E})\in \gpu$ is
obtained from $\Gamma$ by replacing every weight $w_{ij}$ by
\[\widetilde{w}_{ij}=\frac{w_{ij}}{d_i^{\mathrm{in}}}\phi(i) + \frac{w_{ji}}{d_j^{\mathrm{in}}}\phi(j).\]
\end{defi}
Since the weights of the edges are nonnegative and the function
$\phi$ is positive, $\widetilde{\Gamma}\in\gpu$ is an undirected
graph with nonnegative weights. The degree $\widetilde{d}_i$ of
any vertex $i\in V$ in the new graph $\widetilde{\Gamma}$ is given
by \be \label{D51} \widetilde{d}_i = \sum_j\widetilde{w}_{ij}
=\sum_j \frac{w_{ij}}{d_i^{\mathrm{in}}}\phi(i) + \sum_j
\frac{w_{ji}}{d_j^{\mathrm{in}}}\phi(j) = 2\phi(i),\qe where we
used the definition of the in-degree $d_i^\mathrm{in}$ and
\rf{D80}.
\begin{theo}\label{D53}
Let $\Gamma\in\gp$ be an strongly connected graph, then
\[\min_{i\neq 0}\lambda_i(\D(\widetilde{\Gamma}))\leq \min_{i\neq0}\Re(\lambda_i(\D(\Gamma)))\leq
\max_i\Re(\lambda_i(\D(\Gamma)))\leq
\max_{i}\lambda_i(\D(\widetilde{\Gamma})).\]
\end{theo}\begin{proof}For ease of notation we set
$\widetilde{\D}=\D(\widetilde{\Gamma})$ and
$\widetilde{\lambda_i}=\lambda_i(\D(\widetilde{\Gamma}))$. We
consider the inner product for functions $f,g\in
C(\widetilde{V})$,
\[(f,g)=\sum_i\widetilde{d}_i\overline{f(i)}g(i),\] where $\overline{f(i)}$
denotes complex conjugation.  Using \rf{D51}, we obtain the
following identity:
\begin{eqnarray*}
(f,\widetilde{\D}f) &=&
\sum_i\widetilde{d}_i\overline{f(i)}[f(i)-\frac{1}{\widetilde{d}_i}\sum_j\widetilde{w}_{ij}f(j)]\\&=&(f,f)
-
\sum_{i,j}\frac{w_{ij}}{d_i^{\mathrm{in}}}\phi(i)\overline{f(i)}f(j)
-
\sum_{i,j}\frac{w_{ji}}{d_j^{\mathrm{in}}}\phi(j)\overline{f(i)}f(j)
\\&=&
(f,f) -
\sum_{i}\frac{\widetilde{d}_i}{2}\overline{f(i)}\sum_j\frac{w_{ij}}{d_i^{\mathrm{in}}}f(j)
-\sum_{j}\frac{\widetilde{d}_j}{2}f(j)\sum_i\frac{w_{ji}}{d_j^{\mathrm{in}}}\overline{f(i)}\\&=&
(f,f) - \frac{1}{2}(f,P f)-\frac{1}{2}(\overline{f},P\overline{f})
\end{eqnarray*}
Let $u_k$ and $\gamma_k$, $k=0,\ldots,n-1$ be the eigenfunctions
and the corresponding eigenvalues of $P$. Without loss of
generality, we assume that $u_0$ is given by the constant function
$\mathbf{1}=(1,\ldots,1)^\top$ and $\gamma_0=1$. Suppose for the
moment that $(u_k,\mathbf{1})= (u_k,u_0)=0$ for all $k\neq 0$.
Since $\widetilde{\Gamma}\in \gpu$ we can use the usual
variational characterization of the eigenvalues. For all $k\neq 0$
we have
\begin{eqnarray*}
\widetilde{\lambda}_1&=&
\inf_{f\bot\mathbf{1}}\frac{(f,\widetilde{\D}f)}{(f,f)} \leq
\frac{(u_k,\widetilde{\D}u_k)}{(u_k,u_k)}\\&=&
\frac{(u_k,u_k)}{(u_k,u_k)}-\frac{1}{2}\frac{(u_k,P
u_k)}{(u_k,u_k)}-\frac{1}{2}\frac{(\overline{u_k},P
\overline{u_k})}{(u_k,u_k)}\\&=&1 -\frac{1}{2}\gamma_k
-\frac{1}{2}\overline{\gamma_k} = 1-\Re(\gamma_k)= \Re(\lambda_k),
\end{eqnarray*}where we used the fact that if $u_k$ is an eigenfunction for
the eigenvalue $\gamma_k$ then $\bar{u}_k$ is an eigenfunction for
the eigenvalue $\bar{\gamma}_k$. Similarly, we obtain for the
largest eigenvalue $\widetilde{\lambda}_{n-1}$
\begin{eqnarray*}\widetilde{\lambda}_{n-1}= \sup_{f\neq0}\frac{(f,\widetilde{\D}f)}{(f,f)}
\geq \frac{(u_k,\widetilde{\D}u_k)}{(u_k,u_k)}=
 \Re(\lambda_k)\end{eqnarray*} for all $k$.
Therefore, it only remains to show that $(u_k,\mathbf{1})=0$ for
all $k\neq 0$. The Perron-Frobenius Theorem implies that
$\rho=\gamma_0=1$ is a simple eigenvalue of $P$ and hence
$\gamma_k<1$ for all $k\neq 0$. Using \rf{D80} and \rf{D51} we
obtain
\begin{eqnarray*}
(u_k,\mathbf{1}) &=&
\sum_i\widetilde{d}_iu_k(i)\\&=&\sum_i2\phi(i)u_k(i)\\&=&\sum_i2
\sum_j\frac{w_{ji}}{d_j^\mathrm{in}}\phi(j) u_k(i)\\&=& 2\sum_j
\phi(j) \sum_i \frac{w_{ji}}{d_j^\mathrm{in}} u_k(i)\\&=& 2 \sum_j
\phi(j) \gamma_k u_k(j)
\end{eqnarray*}This implies that \[(2-2\gamma_k)\sum_i\phi(i)u_k(i)=0.\]
Since $\gamma_k< 1$ if $k\neq 0$, we conclude that
$\sum_i\phi(i)u_k(i)=0$ and hence $(u_k,\mathbf{1})=0$. This
completes the proof.
\end{proof}
By combining Lemma \ref{Eigenvalues} with Theorem \ref{D53}, we
immediately obtain the following eigenvalue estimates:
\begin{theo}\label{D55}
Let $\Gamma\in\gp$ be a strongly connected graph, then
\[0\leq 1-\sqrt{1-h^2(\widetilde{\Gamma})}\leq \min_{i\neq 0}\Re(\lambda_i(\D(\Gamma)))\leq
\max_i\Re(\lambda_i(\D(\Gamma)))\leq
1+\sqrt{1-(1-\overline{h}(\widetilde{\Gamma}))^2}\leq 2,\] where
$h(\widetilde{\Gamma})$ and $\overline{h}(\widetilde{\Gamma})$ are
the Cheeger constant and the dual Cheeger constant of the graph
$\widetilde{\Gamma}$.
\end{theo}
\begin{rem}
The estimates in Theorem \ref{DAB} are in particular true for
graphs with both positive and negative weights. In contrast, the
estimates in Theorem \ref{D53} only hold for graphs with
nonnegative weights. However, the assumption in Theorem \ref{D53}
that the graph is strongly connected is weaker than the assumption
in Theorem \ref{DAB} that the graph is balanced. Indeed, it is
easy to show that every balanced graph is strongly connected but
not vice versa.
\end{rem}
\subsection{Further eigenvalue estimates}
In the last section, we derived eigenvalue estimates for directed
graphs by using different comparison theorems for directed and
undirected graphs. In this section, we prove further eigenvalue
estimates that do not make use of comparison theorems. By
considering the trace of $\D^2$, we obtain estimates for the
absolute values of the real and imaginary part of the eigenvalues.
\begin{theo} Let $\Gamma \in\g$ be a graph. Then,
\begin{eqnarray*}\min_{i:\lambda_i\neq 0}|\Re(\lambda_i)|&\leq& \sqrt{ \frac{|V_R|
+\sum_{i\in V_R}\left(\frac{w_{ii}^2}{(d^{\mathrm{in}}_i)^2}-
2\frac{w_{ii}}{d^{\mathrm{in}}_i}\right)+ 2\sum_{(i,j)\in
U}\left(\frac{w_{ij}w_{ji}}{d^{\mathrm{in}}_id^{\mathrm{in}}_j}\right)+
\sum_{i=m_0}^{n-1} \Im(\lambda_i)^2}{n-m_0}}\\ &\leq&
\max_i|\Re(\lambda_i)|\end{eqnarray*} where $U\subseteq V_R\times
V_R$ is the set of distinct mutually connected vertices that are
not quasi-isolated, i.~e.~$(i,j)\in U$, if $i\neq j$, and
$d_i^{\mathrm{in}}, d_j^{\mathrm{in}}, w_{ij}, w_{ji}\neq 0$. As
before, $m_0$ denotes the multiplicity of the eigenvalue zero of
$\D$.
\end{theo}
Note that for undirected graphs, the set $U$ is a subset of the
edge set $E$. In particular, if $V_R=V$, and there are no loops in
the graph then $U=E$.
\begin{proof} First, we note that the trace of $\D^2$
satisfies
\begin{equation} \label{D3} \mathrm{Tr}\left(\D^2\right) =\mathrm{Tr}\left(\D_R^2\right)=
\sum_{i=0}^{n-1}\lambda_i^2 = \sum_{i=m_0}^{n-1}\lambda_i^2 =
\sum_{i=m_0}^{n-1}\Re(\lambda_i)^2 -\sum_{i=m_0}^{n-1}
\Im(\lambda_i)^2,\end{equation} where the last equality in
(\ref{D3}) follows from the observation that the eigenvalues
appear in complex conjugate pairs.  An immediate consequence of
Eq.~(\ref{D3}) is:
\begin{equation} \label{D6} (n-m_0)\left(\min_{i:\lambda_i\neq0}|\Re(\lambda_i)|\right)^2 \leq
\mathrm{Tr}\left(\D_R^2\right) + \sum_{i=m_0}^{n-1}
\Im(\lambda_i)^2 \leq
(n-m_0)\left(\max_{i}|\Re(\lambda_i)|\right)^2
\end{equation}On the other hand, the
trace of $\D_R^2$ is given by:
\begin{eqnarray}\label{D4}
\mathrm{Tr}\left(\D_R^{2}\right)&=& \mathrm{Tr}(I_R) -
2\mathrm{Tr}(D_R^{-1}W_R) +\mathrm{Tr}((D_R^{-1}W_R)^2)\nonumber
\\&=& |V_R|-2\sum_{i\in V_R}\frac{w_{ii}}{d^{\mathrm{in}}_i}
+ \sum_{i\in V_R}\left(\frac{w_{ii}}{d^{\mathrm{in}}_i}\right)^2+
\sum_{i,j\in V_R, i\neq
j}\frac{w_{ij}}{d^{\mathrm{in}}_i}\frac{w_{ji}}{d^{\mathrm{in}}_j}\nonumber
\\&=& |V_R| -2\sum_{i\in V_R}\frac{w_{ii}}{d^{\mathrm{in}}_i}+ \sum_{i\in V_R}\left(\frac{w_{ii}}{d^{\mathrm{in}}_i}\right)^2+
 2\sum_{(i,j) \in U}\frac{w_{ij}}{d^{\mathrm{in}}_i}\frac{w_{ji}}{d^{\mathrm{in}}_j}.\end{eqnarray}
 Combining
(\ref{D6}) and (\ref{D4}) completes the proof.
\end{proof}

From this theorem, we can derive interesting special cases.
\begin{coro}\label{D126} If there are no loops and no mutually connected vertices in $V_R$,
i.~e.~ $w_{ii}=0$ for all $i$, and  $U = \emptyset$, then
\[\min_{i:\lambda_i\neq 0}|\Re(\lambda_i)|\leq \sqrt{\frac{ |V_R| +
\sum_{i=m_0}^{n-1} \Im(\lambda_i)^2}{n-m_0}} \leq
\max_i|\Re(\lambda_i)|.
\]
\end{coro}
\begin{coro}\label{D127}Let $\Gamma$ be a loopless, undirected, unweighted, and regular graph,
i.~e.~$w_{ij} \in \{0,1\}$, $w_{ij} = w_{ji}$, and  $d_i =
\sum_jw_{ij} = k \;,\forall\, i\in V$, then
\[ \min_{i\neq 0}\lambda_i \leq \sqrt{\frac{|V| + \frac{2}{k^2}|E|}{n-1}}= \sqrt{\frac{n(k +1)}{({n-1})k}} \leq
 \max_i\lambda_i.\]
\end{coro}The next example shows that this estimate is sharp for
complete graphs.
\begin{example} For a complete graph on $n$ vertices the estimate
in Corollary \ref{D127} yields
\[\min_{i\neq 0}\lambda_i \leq \frac{n}{{n-1}} \leq
\max_i\lambda_i.\] On the other hand, all non-zero eigenvalues of
a complete graph are given by $\frac{n}{n-1}$. Hence, the estimate
in Corollary \ref{D127} is sharp for complete graphs.
\end{example}

In the same way, we can obtain bounds for the absolute values of
the imaginary parts.
\begin{theo}
\begin{eqnarray*}\min_{i:\lambda_i\neq 0}|\Im(\lambda_i)|
&\leq& \sqrt{\frac{\sum_{i=m_0}^{n-1}\Re(\lambda_i)^2
-2\sum_{(i,j)\in
U}\left(\frac{w_{ij}w_{ji}}{d^{\mathrm{in}}_id^{\mathrm{in}}_j}\right)
+\sum_{i\in
V_R}\left(2\frac{w_{ii}}{d^{\mathrm{in}}_i}-\frac{w_{ii}^2}{(d^{\mathrm{in}}_i)^2}\right)-|V_R|}{n-m_0}}
\\&\leq& \max_i|\Im(\lambda_i)|
\end{eqnarray*}
\end{theo}

We obtain the following special case:
\begin{coro} If there are no loops and no mutually connected vertices in $V_R$,
i.~e.~$w_{ii}=0$ for all $i$, and $U = \emptyset$, then
\[\min_{i:\lambda_i\neq 0}|\Im(\lambda_i)|
\leq \sqrt{\frac{\sum_{i=m_0}^{n-1} \Re(\lambda_i)^2-
|V_R|}{n-m_0}} \leq \max_i|\Im(\lambda_i)|.
\]
\end{coro}

\section{Neighborhood graphs}\label{D76}
In \cite{BauerJost} we introduced the concept of neighborhood
graphs for undirected graphs $\Gamma\in \gpu$. Here, we generalize
this concept to directed graphs $\Gamma\in \g$ without
quasi-isolated vertices. As already mentioned above, for the
concept of neighborhood graphs it is crucial to study graphs with
loops. Hence, we will consider graphs with loops in this section.

\begin{defi}
Let $\Gamma =(V,E)\in \g$  and assume that
$d^{\mathrm{in}}_i\neq0$ for all $i\in V$. The neighborhood graph
$\Gamma[l] = (V,E[l])$ of order $l\geq2$ is the graph on the same
vertex set $V$ and its edge set $E[l]$ is defined in the following
way: The weight $w_{ij}[l]$ of the edge from vertex $j$ to vertex
$i$ in $\Gamma[l]$ is given by
\[w_{ij}[l] = \sum_{k_1,\ldots, k_{l-1}}\frac{1}{d^{\mathrm{in}}_{k_1}}\ldots
\frac{1}{d^{\mathrm{in}}_{k_{l-1}}}w_{ik_1}w_{k_1k_2}\ldots
w_{k_{l-1}j}.\]In particular, $j$ is a neighbor of $i$  in
$\Gamma[l]$ if there exists at least one directed path of length
$l$ from $j$ to $i$ in $\Gamma$.
\end{defi}
Another way to look at the neighborhood graph is the following.
The neighborhood graph $\overline{\Gamma}[l]$ of the reversal
graph $\overline{\Gamma}$ encodes the transition probabilities of
a $l$-step random walk on $\overline{\Gamma}$. For a more detailed
discussion of this probabilistic point of view, we refer the
reader to \cite{BJL}.

The neighborhood graph $\Gamma[l]$ has the following properties:
\begin{lemma}\begin{itemize}\item[]
\item[(i)] The in-degrees of the vertices in $\Gamma$ and
$\Gamma[l]$ satisfy
\[d^{\mathrm{in}}_i = d^{\mathrm{in}}_i[l] \quad \forall i\in V \mbox{ and } l\geq 2.\]
\item[(ii)] If $\Gamma$ is balanced, then so is $\Gamma[l]$ and
the out-degrees of the vertices in  $\Gamma$ and $\Gamma[l]$
satisfy
\[d_i^{\mathrm{out}}= d_i^{\mathrm{out}}[l]
\quad \forall i\in V \mbox{ and } l\geq 2.\]
\end{itemize}
\end{lemma}
\begin{proof}$(i)$ We have \begin{eqnarray*}d^{\mathrm{in}}_i[l]&=&\sum_jw_{ij}[l]
 = \sum_{k_1,\ldots, k_{l-1}}\frac{1}{d^{\mathrm{in}}_{k_1}}\ldots
\frac{1}{d^{\mathrm{in}}_{k_{l-1}}}w_{ik_1}w_{k_1k_2}\ldots
w_{k_{l-2}k_{l-1}}\sum_jw_{k_{l-1}j}\\&=&\sum_{k_1,\ldots,
k_l-2}\frac{1}{d^{\mathrm{in}}_{k_1}}\ldots
\frac{1}{d^{\mathrm{in}}_{k_{l-2}}}w_{ik_1}w_{k_1k_2}\ldots
\sum_{k_{l-1}}w_{k_{l-2}k_{l-1}}\\&\vdots&\\&=&\sum_{k_1}
w_{ik_1}=d^{\mathrm{in}}_i.\end{eqnarray*} $(ii)$ Since $\Gamma$
is balanced, we have $d_i^{\mathrm{out}}=d_i^{\mathrm{in}}$ for
all $i\in V$ and thus
\begin{eqnarray*}d_i^{\mathrm{out}}[l] = \sum_jw_{ji}[l] &=&
\sum_{k_1,\ldots,k_{l-1}}\frac{1}{d^{\mathrm{in}}_{k_1}}\ldots
\frac{1}{d^{\mathrm{in}}_{k_{l-1}}}w_{k_1k_2}\ldots
w_{k_{l-2}k_{l-1}}w_{k_{l-1}i}\sum_jw_{jk_1}\\&=&
\sum_{k_2,\ldots,k_{l-1}}\frac{d^{\mathrm{out}}_{k_1}}{d^{\mathrm{in}}_{k_1}}\frac{1}{d^{\mathrm{in}}_{k_2}}\ldots
\frac{1}{d^{\mathrm{in}}_{k_{l-1}}}w_{k_2k_3}\ldots
w_{k_{l-2}k_{l-1}}w_{k_{l-1}i}\sum_{k_1}w_{k_1k_2}\\&\vdots&\\&=&\sum_{k_{l-1}}w_{k_{l-1}i}=
d_i^{\mathrm{out}}.
\end{eqnarray*} Consequently, if $\Gamma$ is balanced, then we have for all $i$, $d_i^\mathrm{in}[l]=
d_i^\mathrm{in}= d_i^\mathrm{out}=d_i^\mathrm{out}[l]$ and hence
$\Gamma[l]$ is balanced.
\end{proof}

The next theorem establishes the relationship between $\D$ and
$\D[l]$.
\begin{theo}\label{D142} We have
\be \label{D143} I-(I-\Delta)^{l}= I-P^{l}= \Delta[l],\qe where
$\Delta[l]$ is the graph Laplace operator on $\Gamma[l]$ and
$\Delta$ is the graph Laplace operator on $\Gamma$.
\end{theo}The proof is essentially the same as the proof given in
\cite{BauerJost} for undirected graphs. So we omit the details
here.
\begin{coro}The multiplicity $m_1$ of the eigenvalue one
is an invariant for all neighborhood graphs, i.~e.~$m_1(\Delta) =
m_1(\Delta[l])$ for all $l\geq 2$.
\end{coro}
\begin{proof}$\Gamma$ and $\Gamma[l]$ have the same vertex set, thus
both $\Delta$ and $\Delta[l] = I-(I-\Delta)^{l}$ have $n=|V|$
eigenvalues. By Theorem \ref{D142}, every eigenfunction $u_k$ for
$\Delta$ and eigenvalue $\lambda_k$ is also an eigenfunction for
$\Delta[l]$ and eigenvalue $1-(1-\lambda_k)^l$. Thus, the
corollary follows from the observation that $1-(1-\lambda_k)^l=1$
iff $\lambda_k =1$.
\end{proof}
As in \cite{BauerJost}, the relationship between the spectrum of a
graph and the spectrum of its neighborhood graphs can be exploited
to derive new eigenvalue estimates. For example we have the
following result:\begin{theo}\label{D61} Let $\Gamma$ be a graph
and $\Gamma[l]$ be its neighborhood graph of order $l\geq2$.
\begin{itemize} \item[$(i)$]If $1\leq \mathcal{A}[l]\leq
\min_{i\neq0}|\lambda_i[l]|$, then
$(\mathcal{A}[l]-1)^{\frac{1}{l}}\leq |1-\lambda_i|$ for all
$i\neq 0$, where $\mathcal{A}[l]$ is any lower bound for
$\min_{i\neq0}|\lambda_i[l]|$.  \item[$(ii)$] If $\min_{i\neq 0}
|\lambda_i[l]|\leq \mathcal{B}[l]\leq 1$, then
$(1-\mathcal{B}[l])^{\frac{1}{l}}\leq \max_i|1-\lambda_i|$, where
$\mathcal{B}[l]$ is any upper bound for $\min_{i\neq 0}
|\lambda_i[l]|$. \item[$(iii)$] If $1\leq \mathcal{C}[l]\leq
\max_i|\lambda_i[l]|$, then $(\mathcal{C}[l]-1)^\frac{1}{l}\leq
\max_i|1-\lambda_i|$, where $\mathcal{C}[l]$ is any lower bound
for $\max_i|\lambda_i[l]|$.\item[$(iv)$]If
$\max_i|\lambda_i[l]|\leq \mathcal{D}[l]\leq 1$, then
$(1-\mathcal{D}[l])^\frac{1}{l}\leq |1-\lambda_i|$ for all $i$,
where $\mathcal{D}[l]$ is any upper bound for
$\max_i|\lambda_i[l]|$.
\end{itemize}
\end{theo}
\begin{proof}
$(i)$. From Theorem \ref{D142} we have $\lambda_i[l]=
1-(1-\lambda_i)^l$. Thus, we have for all $i\neq 0$
\[\mathcal{A}[l]\leq |1-(1-\lambda_i)^l| \leq 1+|1-\lambda_i|^l,
\] where we used the triangle inequality.\\
$(ii)$. We have
\[\mathcal{B}[l]\geq \min_{i\neq0}|1-(1-\lambda_i)^l)|\geq
1-(\max_i|1-\lambda_i|)^l,\] where we used the reverse triangle
inequality.
\\ $(iii)$.  We have \[\mathcal{C}[l] \leq
\max_i|1-(1-\lambda_i)^l|\leq 1+(\max_i|1-\lambda_i|)^l,\] where
we used again the triangle inequality. \\ $(iv)$. For all $i$ we
have \[\mathcal{D}[l] \geq |1-(1-\lambda_i)^l| \geq 1 -
|1-\lambda_i|^l,\]where we used again the reverse triangle
inequality.
\end{proof}
One can exploit the neighborhood graph technique further. For
instance, by using similar arguments as in \cite{BauerJost} one
can obtain estimates for $\Re(\lambda_i)$ and $|\Im(\lambda_i)|$.

\end{document}